%% file: main.tex
\documentclass{article}

\usepackage{arxiv}

\usepackage[utf8]{inputenc} % allow utf-8 input
\usepackage[T1]{fontenc}    % use 8-bit T1 fonts
\usepackage{url}            % simple URL typesetting
\usepackage{booktabs}       % professional-quality tables
\usepackage{amsfonts}       % blackboard math symbols
\usepackage{nicefrac}       % compact symbols for 1/2, etc.
\usepackage{microtype}      % microtypography
\usepackage{lipsum}
\usepackage{verbatim}       % Multiline comments

\RequirePackage{amssymb}    % Extra symbols
\RequirePackage{amsthm}     % Theorem-like environments
\RequirePackage{thmtools}   % Theorem-like environments, extends amsthm
\RequirePackage{mathtools}  % Fonts and environments for mathematical formulae
\RequirePackage{mathrsfs}   % Script font with \mathscr{}
\RequirePackage{cancel}     % Cancel terms with \cancel, \bcancel or \xcancel
\RequirePackage{stmaryrd}   % Brackets

\RequirePackage{tikz}             % Drawing tool
\usetikzlibrary{calc}
\usetikzlibrary{intersections}
\usetikzlibrary{decorations.markings}
\usetikzlibrary{cd} % An implementation of simple commutative diagrams in tikz
\RequirePackage[all]{xy} % For backwards compatibility with respect to publishing articles.

\RequirePackage{varioref}                                     % \vref
\RequirePackage{hyperref}                                     % Clickable links
\RequirePackage[nameinlink, capitalize, noabbrev]{cleveref}   % \cref
\RequirePackage{doi}           % Ignore LaTeX syntax in DOI links

\declaretheoremstyle[headfont   = \bfseries\sffamily,
                     notefont   = \normalfont,
                     spaceabove = 6pt plus 0pt minus 2pt]{plain}
\declaretheoremstyle[headfont   = \bfseries\sffamily,
                     notefont   = \normalfont,
                     spaceabove = 6pt plus 0pt minus 2pt]{definition}
\declaretheorem[style = plain, numberwithin = section]{theorem}
\declaretheorem[style = plain,      sibling = theorem]{corollary}
\declaretheorem[style = plain,      sibling = theorem]{lemma}
\declaretheorem[style = plain,      sibling = theorem]{proposition}

\declaretheorem[style = definition, sibling = theorem]{definition}
\declaretheorem[style = definition, sibling = theorem]{example}
\declaretheorem[style = definition, sibling = theorem]{notation}
\declaretheorem[style = remark,     sibling = theorem]{remark}

\crefname{observation}{Observation}{Observations}
\Crefname{observation}{Observation}{Observations}
\crefname{conjecture}{Conjecture}{Conjectures}
\Crefname{conjecture}{Conjecture}{Conjectures}
\crefname{notation}{Notation}{Notations}
\Crefname{notation}{Notation}{Notations}
\crefname{paper}{Paper}{Papers}
\Crefname{paper}{Paper}{Papers}

\newcommand{\x}{\mathbf{x}}

\title{Homotopy Theory of Non-singular Simplicial Sets}

\author{
  Vegard Fjellbo \\
  Department of Mathematics \\
  University of Oslo \\
  Oslo, Norway \\
  \texttt{vegard.fjellbo@gmail.com} \\
   \And
}

\begin{document}
\maketitle

\input{sections/abstract.tex}

\input{sections/intro}

\input{sections/preexist}

\input{sections/structure}

\input{sections/behavior}

\input{sections/planes}

\input{sections/properties}

\input{sections/lifting}

\input{sections/cofib}

\input{sections/inverse}

\input{sections/relations}

% keywords can be removed

\bibliographystyle{unsrt}  
\bibliography{main}  %%% Remove comment to use the external .bib file (using bibtex).

\end{document}

%% file: sections/abstract.tex
\begin{abstract}
\noindent A simplicial set is said to be \textbf{non-singular} if its non-degenerate simplices are embedded. Let $sSet$ denote the category of simplicial sets. We prove that the full subcategory $nsSet$ whose objects are the non-singular simplicial sets admits a model structure such that $nsSet$ is Quillen equivalent to $sSet$ equipped with the standard model structure due to Quillen \cite{Qu67}. The model structure on $nsSet$ is right-induced from $sSet$ and it makes $nsSet$ a proper cofibrantly generated model category. Together with Thomason's model structure on small categories \cite{Th80} and Raptis' model structure on posets \cite{Ra10} these form a square-shaped diagram of Quillen equivalent model categories in which the subsquare of right adjoints commutes.
\end{abstract}

%% file: sections/intro.tex
\section{Introduction}
\label{sec:intro_hty}

\noindent This paper concerns the diagram
\begin{equation}
\label{eq:diagram_of_adjunctions}
\begin{gathered}
\xymatrix{
&& sSet \ar@<-.7ex>[ld]_{Sd^2} \\
Cat \ar@<-.7ex>[d]_p \ar@<-.7ex>[r]_N & sSet \ar@<-.7ex>[ur]_{Ex^2} \ar@<-.7ex>[l]_c \ar@<-.7ex>[d]_D \\
PoSet \ar@<-.7ex>[u]_U \ar@<-.7ex>[r]_N & nsSet \ar@<-.7ex>[l]_q \ar@<-.7ex>[u]_U
}
\end{gathered}
\end{equation}
which will be properly explained in \cref{sec:pre_exist}. For now it suffices to say the following.

The diagram (\ref{eq:diagram_of_adjunctions}) consists of adjunctions between categories, where $sSet$ is the category of simplicial sets, where $Cat$ is the category of small categories, where $PoSet$ is the full subcategory of $Cat$ whose objects are the partially ordered sets (posets) and where $nsSet$ is the category of non-singular simplicial sets. The (full) inclusion $U:nsSet\to sSet$ admits a right adjoint functor \cite[Rem.~2.2.12]{WJR13}, which is known as desingularization and denoted $D$.

Due to the preexisting literature, all of the categories that appear in (\ref{eq:diagram_of_adjunctions}), except $nsSet$, are model categories. Furthermore, all of the adjunctions that appear, except $(D,U)$ and $(q,N)$, are Quillen equivalences. The aim of this paper is to establish a model structure on $nsSet$ such that $(D,U)$ and $(q,N)$ are Quillen equivalences. This is essentially a reformulation of \cref{thm:main_homotopy_theory} below, which is our main result.

For a justification of the model structure on $nsSet$ that we here suggest, see the highlight that is \cref{lem:Pushout_along_strom_homotopically_wellbehaved} and its implication \cref{lem:unit_weak_eq}, which says that the unit of the adjunction $(DSd^2,Ex^2U)$ is a weak equivalence.

Given a simplicial set $X$, there is --- according to the Yoneda lemma --- a natural bijection $x\mapsto \bar{x}$ from the set $X_n$ of $n$-simplices to the set $sSet(\Delta [n],X)$ of simplicial maps from the standard $n$-simplex $\Delta [n]$ to $X$.
\begin{definition}\label{def:embedded_simplex}
Let $X$ be a simplicial set. The map $\bar{x}$ that corresponds to a simplex $x$ of $X$ under the natural bijection
\[X_n\xrightarrow{\cong } sSet(\Delta [n],X)\]
given by $x\mapsto \bar{x}$ is the \textbf{representing map} of $x$. A simplex is \textbf{embedded} if its representing map is degreewise injective.
\end{definition}
\noindent The terminology of \cref{def:embedded_simplex} makes sense of the notion of non-singular simplicial set. Here, we follow the terminology of Waldhausen, Jahren and Rognes \cite[Def.~1.2.2, p.~14]{WJR13}.

In the diagram, the functor $Sd:sSet\to sSet$ is the Kan subdivision \cite[p.~147]{FP90} and $Ex$ denotes its right adjoint \cite[Prop.~4.2.10]{FP90}, which is sometimes referred to as extension \cite[p.~212]{FP90}. The symbol $Sd^k$, for $k\geq 0$, simply denotes the $k$-fold iteration of $Sd$, so in particular the symbol $Sd^2$ means the composite of $Sd$ with itself. Similarly, the symbol $Ex^2$ denotes the functor that performs extension twice.

There is a standard model structure on $sSet$ due to Quillen \cite{Qu67} in which the weak equivalences are the maps whose geometric realizations are (weak) homotopy equivalences, the fibrations are the Kan fibrations and the cofibrations are the degreewise injective maps. Regarding the terminology of the theory model categories, we follow Hirschhorn's book \cite{Hi03}, but we also refer to Hovey's book \cite{Ho99}, which differs only slightly from the former. The differences are explained whenever relevant.

In the passage between the categories $sSet$ and $nsSet$, there is a homotopical issue, namely that desingularization does not in general preserve the homotopy type, though every simplicial set is cofibrant in the standard model structure. We will discuss the issue in \cref{sec:behavior}. Nevertheless, we will prove the following result.
\begin{theorem}\label{thm:main_homotopy_theory}
Equip $sSet$ with the standard model structure. There is a proper, cofibrantly generated model structure on $nsSet$ such that $f$ is a weak equivalence (resp. fibration) if and only if $Ex^2U(f)$ is a weak equivalence (resp. fibration), and such that
\[DSd^2:sSet\rightleftarrows nsSet:Ex^2U\]
is a Quillen equivalence.
\end{theorem}
\noindent This theorem is our main result. Note that $Ex^2U(f)$ is a weak equivalence if and only if $U(f)$ is a weak equivalence, as $Ex$ preserves and reflects weak equivalences \cite[Cor.~4.6.21]{FP90}. Moreover, we will in \cref{sec:relations} argue that each adjunction that appears in (\ref{eq:diagram_of_adjunctions}) is a Quillen equivalence.

Notice that non-singular simplicial sets is an intermediate between ordered simplicial complexes and simplicial sets in the following sense. In an ordered simplicial complex, the vertices of every simplex are pairwise distinct. Moreover, every simplex is uniquely determined by its vertices. In a non-singular simplicial set, the vertices of every non-degenerate simplex are pairwise distinct. However, a simplex is not necessarily uniquely determined by its vertices. In an arbitrary simplicial set, the vertices of a non-degenerate simplex are not necessarily pairwise distinct.

Moreover, $nsSet$ as a category is strictly between ordered simplicial complexes and $sSet$. This is automatic from the definition of $nsSet$ as a full subcategory of $sSet$, because every simplicial set associated with an ordered simplicial complex is non-singular. Making $nsSet$ a model category puts the homotopy theory of ordered simplicial complexes more directly into the modern context of model categories.

An advantage of non-singular simplicial sets over simplicial sets is that the former have a natural PL structure described in \cite[Sec.~3.4,~p.~126--127]{WJR13}. The key to this fact is the compatibility between the Kan subdivision of simplicial sets and the barycentric subdivision of simplicial complexes. The former performed on a non-singular simplicial set is (associated with) an ordered simplicial complex. See the explanation on page 36 in the book by Waldhausen, Jahren and Rognes \cite{WJR13} and Lemmas 2.2.10. and 2.2.11. \cite[p.~38]{WJR13} in the same book. The category $nsSet$ plays an important role there. Compared with ordered simplicial complexes, the category of non-singular simplicial sets has colimits that are somewhat more meaningful in the sense that more of the colimits are preserved by geometric realization.

In \cref{sec:pre_exist}, we properly introduce the diagram (\ref{eq:diagram_of_adjunctions}). \cref{sec:structure} explains our chosen method for establishing the model structure on $nsSet$.

Sections \ref{sec:behavior} throughout \ref{sec:lifting} concern the proof of \cref{prop:main_homotopy_theory}, which says that $nsSet$ is a cofibrantly generated model category and that $(DSd^2,Ex^2U)$ is a Quillen pair. Towards a proof of this, \cref{sec:behavior} begins by discussing the intution behind \cref{thm:main_homotopy_theory} and its connection to regular neighborhood theory. On that note, we introduce the important notion of Str\o m map whose properties are discussed in \cref{sec:properties}. The Str\o m maps form a class of auxiliary morphisms, which is used as a tool to establish the announced model structure on $nsSet$. \cref{sec:planes} handles important technicalities in that it shows how desingularization behaves when applied to certain pushouts. In \cref{sec:lifting}, we verify that the criteria laid out in \cref{sec:structure} are indeed satisfied so that \cref{prop:main_homotopy_theory} holds.

We discuss cofibrations in \cref{sec:cofib} and state and prove \cref{prop:axiom_of_propriety}, which is the axiom of propriety. The sole purpose of \cref{sec:inverse} is to prove that $(DSd^2,Ex^2U)$ is a Quillen equivalence, which is stated as \cref{prop:homotopy_inverse}. \cref{thm:main_homotopy_theory} then immediately follows.

Finally, in \cref{sec:relations}, we fullfill our promise that every adjunction in the diagram (\ref{eq:diagram_of_adjunctions}) is a Quillen equivalence.

%% file: sections/preexist.tex
\section{Pre-existing model structures}
\label{sec:pre_exist}

We will explain the aspects of the diagram (\ref{eq:diagram_of_adjunctions}) that were not explained in \cref{sec:intro_hty}.

If the inclusion of a full subcategory has a left adjoint, then we will refer to the subcategory as a \textbf{reflective} subcategory. Note that the terminology is not standard. Although the fullness assumption seems more common today than before, Mac Lane's notion \cite{ML98}, for example, does not include fullness as an assumption in his definition. Nor do Adámek and Rosický \cite{AR15} include fullness as an assumption in their notion.

\subsection{Simplicial sets}

We view a \textbf{simplicial set} as a functor $\Delta ^{op}\to Set$ where $\Delta$ is the category of finite ordinals and $\Delta ^{op}$ its opposite. The objects of $\Delta$ are the totally ordered sets
\[[n]=\{ 0<1<\cdots <n\} ,\]
$n\geq 0$, and its morphisms are the order-preserving functions $\alpha :[m]\to [n]$, meaning $\alpha (i)\leq \alpha (j)$ whenever $i\leq j$. We refer to the morphisms as \textbf{operators}. This is because they operate (to the right) on the simplices of a simplicial set. We will write $X_n=X([n])$ for brevity whenever $X$ is a simplicial set. The symbol $sSet$ denotes the category of simplicial sets and natural transformations. To a large extent we follow the notation from Chapter 4 of Fritsch and Piccinini's book ``Cellular Structures in Topology'' \cite{FP90} on the topic of simplicial sets.

Throughout this paper, we will use the following symbols.
\begin{notation}\label{not:prototypes_cofibr_trivial_cofibr_standard_pre_exist}
The elements of the set
\[I=\{\partial \Delta [n]\to \Delta [n]\mid n\geq 0\}\]
of inclusions of boundaries into the standard simplices are prototypes of the cofibrations in $sSet$ equipped with the standard model structure. Similarly, the elements of the set
\[J=\{\Lambda ^k[n]\to \Delta [n]\mid 0\leq k\leq n>0\}\]
of inclusions of horns into the standard simplices are prototypes of the trivial cofibrations.
\end{notation}

\subsection{Passage between simplicial sets and non-singular simplicial sets}

Notice that a product of non-singular simplicial sets is again non-singular, and that a simplicial subset of a non-singular simplicial set is again non-singular \cite[Rem.~2.2.12]{WJR13}. These facts give rise rise to the construction of desingularization.
\begin{definition}{Remark 2.2.12. in \cite[p.~39]{WJR13} }
\label{def:desing}
Let $X$ be a simplicial set. The \textbf{desingularization} of $X$, denoted $DX$, is the image of the map
\[X\rightarrow \prod _{f:X\rightarrow Y}Y\]
given by $x\mapsto (f(x))_f$, where the product is indexed over the quotient maps $f:X\rightarrow Y$ with non-singular target $Y$.
\end{definition}
\noindent The construction $DX$ is functorial and the degreewise surjective map that comes with it is seen to be a natural map $\eta _X:X\to UDX$ \cite[Rem.~2.2.12]{WJR13}.

From the construction in \cref{def:desing}, it follows that any map $X\xrightarrow{f} Y$ whose target $Y$ is non-singular factors through $X\to DX$ \cite[Rem.~2.2.12]{WJR13}. This is because any degreewise surjective map whose source is $X$ and whose target is non-singular can be canonically identified with a quotient map. On the other hand, the factorization is unique because the degreewise surjective maps are precisely the epics of $sSet$. In fact, the natural map $\eta _X$ is the unit of a unit-counit pair $(\eta _X,\epsilon _A)$ \cite[Rem.~2.2.12]{WJR13}. This is also stated as \cite[Lem.~2.2.2.]{Fj20-DN}.

In the language suggested above, the category of non-singular simplicial sets is a reflective subcategory of the category of simplicial sets. Hirschhorn takes as an assumption on his notion of model category that the underlying category is bicomplete \cite[Def.~7.1.3, p.~109]{Hi03}, so we do too. We say that a category is \textbf{bicomplete} if it is complete and cocomplete. A consequence of the fact that $nsSet$ is a reflective subcategory of $sSet$ is that $nsSet$ is bicomplete. An explanation of this fact is provided by \cite[Cor.~2.2.3.]{Fj20-DN}.

\subsection{Thomason's model structure}

The symbol $N$ denotes the nerve functor \cite[p.~106]{Se68}. It takes a small category $\mathscr{C}$ to the simplicial set whose set of $n$-simplices, for each $n\geq 0$, is the set of functors $[n]\to \mathscr{C}$. According to G. Segal \cite[p.~105]{Se68}, the nerve construction appears at least implicitly in the work of Grothendieck. It is well known that $N$ is fully faithful and that it has a left adjoint $c:sSet\to Cat$, called categorification. The fact can be extracted from \cite{GZ67}, according to R. Fritsch and D. M. Latch \cite[p.~147]{FL81}.

Due to Thomason, we can give equip $Cat$ with a right-induced cofibrantly generated model category such that $(cSd^2,Ex^2N)$ is a Quillen equivalence \cite{Th80} whose source is $sSet$ with the standard model structure due to Quillen. Cisinski have made a correction to Thomason's erroneous argument that $Cat$ is proper \cite{Ci99} so that there is one more adjective that one can use.

\subsection{Raptis' model structure}

A \textbf{poset} is a small category such that each hom set consists of at most one element and such that there are no isomorphisms but the identities. Notice that a set equipped with a reflexive, antisymmetric and transitive binary relation $\leq$ can intuitively be viewed as a poset by letting there be a morphism $x\to y$ if and only if $x\leq y$.

We let $U:PoSet\to Cat$ be the inclusion and $p$ its right adjoint. The easiest way to obtain $p$ is probably to consider the category of preorders, which is strictly between $Cat$ and $PoSet$. A small category $\mathscr{C}$ is a \textbf{preorder} if each hom set $\mathscr{C} (c,c')$ has at most one element. Let $PreOrd$ denote the full subcategory of $Cat$ whose objects are the preorders. It is not hard to see that each of the inclusions of the composite
\[PoSet\to PreOrd\to Cat\]
has a left adjoint. In other words, the category of posets is a reflective subcategory of $Cat$. 

Raptis has restricted Thomason's model structure to the category of posets so that $(p,U)$ is a Quillen equivalence \cite{Ra10}.

\subsection{Passage between non-singular simplicial sets and posets}

Overload the symbol $N$ so that it also refers to the corestriction to $nsSet$ of the restriction of $N:Cat\to sSet$ to the subcategory $PoSet$. By this we simply mean the following. If $G:\mathscr{B} \to \mathscr{A}$ is a functor between categories, then the \textbf{image of $F$}, denoted $\textrm{Im} \, F$, is the smallest subcategory of the target $\mathscr{B}$ that contains any object and any morphism that is hit by $G$. If $\mathscr{C}$ is a subcategory of $\mathscr{A}$ that contains $\textrm{Im} \, F$, then we say that the induced functor $\mathscr{B} \to \mathscr{C}$ is the \textbf{corestriction of $G$ to $\mathscr{C}$}.

Define $q=pcU$. As $U:nsSet\to sSet$ is a full inclusion it follows that $q$ is left adjoint to $N:PoSet\to nsSet$. To verify the latter statement, let $G$ in \cref{lem:corestriction_of_rightadjoint_to_full_subcategory} be the composite
\[PoSet\xrightarrow{U} Cat\xrightarrow{N} sSet\]
and let $\mathscr{C} =nsSet$.
\begin{lemma}\label{lem:corestriction_of_rightadjoint_to_full_subcategory}
Any corestriction $\bar{G}$ of a right adjoint $G:\mathscr{B} \to \mathscr{A}$ to a full subcategory $\mathscr{C}$ of its target $\mathscr{A}$ admits a left adjoint. Moreover, a restriction to $\mathscr{C}$ of a choice $F$ of a left adjoint to $G$ is left adjoint to $\bar{G}$.
\end{lemma}
\begin{proof}
Let $U$ denote the inclusion $\mathscr{C} \to \mathscr{A}$. The counit $\epsilon _b:FG(b)\to b$ of the adjunction
\[F:\mathscr{A} \rightleftarrows \mathscr{B} :G\]
is already a natural map $(FU)\bar{G} (b)\to b$ as $FG=F(U\bar{G} )=(FU)\bar{G}$. We let $\bar{\epsilon } _b$ denote this map. If $c$ is an object of $\mathscr{C}$, then we have the unit $\eta _{U(c)}:U(c)\to GF(U(c))$. As $GF(U(c))=(U\bar{G} )F(U(c))=U(\bar{G} FU(c))$ there is a unique map $\bar{\eta } _c:c\to \bar{G} FU(c)$ such that $\eta _{U(c)}=U(\bar{\eta } _c)$. It is straight forward to check that the natural maps $\bar{\eta } _c$ and $\bar{\epsilon } _b$ satisfy the compatibility criteria of a unit and a counit.
\end{proof}
\noindent By design, then, the square of right adjoints in (\ref{eq:diagram_of_adjunctions}) commutes precisely, meaning $N\circ U=U\circ N$.

\subsection{Jardine's subdivision model structures}

J. F. Jardine \cite{Ja13} has established a model structure on $sSet$ that he calls the $Sd^2$-model structure. It is defined in such a manner that $(Sd^2,Ex^2)$ is a Quillen equivalence \cite[Thm.~1.1.,~p.~274]{Ja13} and that $(c,N)$ is a Quillen equivalence \cite[Thm.~3.1.,~p.~286]{Ja13}. The weak equivalences of the $Sd^2$-model structure are the same as the standard ones.

The fibrations and cofibrations of the $Sd^2$-model structure are defined thus. A map $p$ of $sSet$ is an \textbf{$Ex^2$-fibration} if $Ex^2(p)$ is a Kan fibration. To define the cofibrations, we might as well introduce the following standard terminology at this point.
\begin{definition}
Given a solid arrow commutative square
\begin{displaymath}
\xymatrix{
A \ar[d]_i \ar[r] & X \ar[d]^p \\
B \ar[r] \ar@{-->}[ur] & Y
}
\end{displaymath}
in some category, we say that a dashed map $B\to X$ is a \textbf{lifting} if it makes the whole diagram commute. In this case we say that $(i,p)$ is a \textbf{lifting-extension pair}, that $i$ has the \textbf{left lifting property (LLP)} with respect to $p$ and that $p$ has the \textbf{right lifting property (RLP)} with respect to $i$.
\end{definition}
\noindent A map $i$ of $sSet$ is a \textbf{$Sd^2$-cofibration} if $(i,p)$ is a lifting-extension pair for each $Ex^2$-fibration $p$. Because $Ex$ preserves Kan fibrations \cite[Lem.~4.6.15, p.~213]{FP90}, the $Sd^2$-model structure is shifted in the sense that the weak equivalences are the same and that there are more fibrations and less cofibrations.

%% file: sections/structure.tex
\section{Strategy to establish the model structure}
\label{sec:structure}

\noindent We find ourselves in a similar situation as that of Thomason. Prior to his article \cite{Th80} there was a homotopy theory of small categories for which Quillen's paper \cite{Qu67} is a reference. It is thought of as inherited from topological spaces via the classifying space. The nerve induces an equivalence of the homotopy categories, yet its left adjoint $c:sSet\to Cat$ does not induce an (inverse) equivalence.

After the recent development of his time, Thomason discovered that the geometrically favorable construction $cSd^2$ preserves homotopy type \cite{Th80} and managed to put a model structure on small categories that makes it Quillen equivalent to simplicial sets, with $cSd^2$ as the left Quillen functor. Fritsch and Latch \cite{FL81} present a contemporary view of the historical development and explain how surprising the result was.

Similarly, there exists a homotopy theory of ordered simplicial complexes thought of as inherited from simplicial sets. The category of ordered simplicial complexes is slightly smaller than $nsSet$. The inclusion $U:nsSet\to sSet$ is full by definition and has a left adjoint called desingularization, as we explained in \cref{sec:intro_hty}. We will display examples of the behavior of desingularization in \cref{sec:behavior}.

There are two main differences between our situation and that of Thomason, namely that we can build on his work and that desingularization is in some sense more difficult to work with.

Categorification $c:sSet\to Cat$ has the following rather elementary description. For $X$ a simplicial set, let the the set of objects $obj(cX)$ of $cX$ be the set $X_0$ of $0$-simplices. The morphisms are freely generated by the set $X_1$ of $1$-simplices with $x\in X_1$ viewed as a morphism $x\delta _1\to x\delta _0$, and then imposing a composition relation $x\delta _1=x\delta _0\circ x\delta _2$, for all $2$-simplices $x\in X_2$. Here, $\delta _j$ is the elementary face operator that omits the index $j$.

On the other hand, desingularization has the two descriptions given in \cref{def:desing} and \cite[Thm.~2.1.3.]{Fj20-DN}. In general, these can be more difficult to work with. We will essentially be using the latter description, albeit a modification.

The strategy we shall use to obtain the model structure on $nsSet$ is essentially the lifting method that Thomason \cite{Th80} uses, except that it has become standardized. It is summarized in the following theorem, credited to D. M. Kan. The language we use is that of Theorem 11.3.2 in Hirschhorn's textbook \cite[p.~214]{Hi03}.
\begin{theorem}[D.M. Kan]\label{thm:lifting_across_adjunction}
Suppose there is an adjunction
\[F:\mathscr{M} \rightleftarrows \mathscr{N} :G\]
where $\mathscr{M}$ is a cofibrantly generated model category with $I$ as the set of generating cofibrations and $J$ as the set of generating trivial cofibrations. Furthermore, assume that $\mathscr{N}$ is a bicomplete category. If
 \begin{enumerate}
  \item {(First lifting condition) each of the sets $FI$ and $FJ$ permits the small object argument, and}
  \item{(Second lifting condition) $G$ takes relative $FJ$-cell complexes to weak equivalences,}
 \end{enumerate}
then $\mathscr{N}$ is a cofibrantly generated model category where the weak equivalences of $\mathscr{N}$ are the morphisms $f$ such that $Gf$ is a weak equivalences, and where $FI$ and $FJ$ are the generating cofibrations and generating trivial cofibrations, respectively. Moreover, $(F,G)$ becomes a Quillen pair.
\end{theorem}
\noindent Formalities ensure that a morphism $f$ in $\mathscr{N}$ is a fibration in the lifted model structure if and only if $Gf$ is a fibration. The language of \cref{thm:lifting_across_adjunction} is fairly standard, but it will be interpreted or explained to a suitable extent when we get to the relevant part.

We will make use of \cref{thm:lifting_across_adjunction} in order to establish the model structure by considering the case when
\[(F,G)=(DSd^2,Ex^2U)\]
and when $sSet$ has the standard model structure.

Recall \cref{not:prototypes_cofibr_trivial_cofibr_standard_pre_exist}. In our case, $I$ serves as a set of generating cofibrations for $sSet$ and $J$ serves as a set of generating trivial cofibrations for $sSet$. The method of lifting the standard model structure on $sSet$ to $nsSet$ is justified by the fact that $U(f)$ is a weak equivalence if and only if $Ex^2U(f)$ is a weak equivalence.

The key to verifying the second lifting condition is the notion Str\o m map, introduced in \cref{def:strom}. Str\o m maps have good technical properties, as shown by \cref{prop:Strom-maps_closed_under_cobasechange}, and good homotopical properties, as shown by \cref{lem:Pushout_along_strom_homotopically_wellbehaved}. At the same time, the class of Str\o m maps contains the sets $DSd^2(I)$ and $DSd^2(J)$, which \cref{cor:two-fold_subdivision_strom} shows.

%% file: sections/behavior.tex
\section{Homotopical behavior of desingularization}
\label{sec:behavior}

In this section, we display examples of the behavior of desingularization. Specifically, we display the results of desingularizing a few models of spheres. In \cref{sec:inverse}, we explain that the two-fold Kan subdivision $Sd^2$ performed before desingularization ensures that the homotopy type is not altered. This is analogous to Thomason's situation \cite{Th80}. Note that performing the Kan subdivision once before desingularization is not enough.

Forming the colimit of a diagram in $nsSet$ can be done by forgetting that the involved simplicial sets are non-singular, forming the colimit in $sSet$ instead, and finally applying desingularization.

Consider some of the usual models for spheres. It is not hard to realize that
\[D(\Delta [n]/\partial \Delta [n])\cong \Delta [0]\]
for every $n>0$. Not much harder is it to see that
\[DSd(\Delta [n]/\partial \Delta [n])\cong \Delta [1]\]
for every $n>1$. Thus in these cases, desingularization does not preserve homotopy type. Note that the case $n=1$ is special as $Sd(\Delta [1]/\partial \Delta [1])$ is two copies of $\Delta [1]$ glued together along their boundaries. Hence, this simplicial set is already non-singular. So desingularization trivially preserves homotopy type in this case.

The $2$-sphere can be modeled by $X=Sd^2(\Delta [2]/\partial \Delta [2])$. This is because the Kan subdivision preserves colimits \cite[Cor.~4.2.11]{FP90} and degreewise injective maps \cite[Cor.~4.2.9]{FP90}. Hence, the simplicial set $Sd^2(\partial \Delta [2])$ can be considered the boundary of $Sd^2(\Delta [2])$ and the simplicial set $X$ is the result of collapsing this boundary. \cref{fig:ch1_Desing_doublysubd_2simpl} is meant to indicate that $DX$ is the suspension of a $1$-sphere, modelled by a 12-gon, which we have formulated as \cite[Prop.~2.4.4.]{Fj20-DN}. In other words, desingularization preserves the homotopy type in this case. One might attribute the behavior to properties of the inclusion
\[Sd^2(\partial \Delta [2])\to Sd^2(\Delta [2])\]
of the boundary. \cref{def:strom}, \cref{cor:two-fold_subdivision_strom} and \cref{lem:Pushout_along_strom_homotopically_wellbehaved} will make this claim precise. The intuition is that the two-fold subdivision creates a sufficiently nice neighborhood around the boundary.

Here we transfer Thomason's insights \cite[Prop.~4.3]{Th80}, which most likely come from regular neighborhood theory, to our setting. Regular neighborhood theory is treated in the sources \cite[§3]{RS72} and \cite[§II]{Hu69}.

The functor $DSd$ takes the instance
\[d_{\Delta [2]/\partial \Delta [2]}:Sd(\Delta [2]/\partial \Delta [2])\xrightarrow{\sim } \Delta [2]/\partial \Delta [2]\]
of the last vertex map, which is in general a weak equivalence, to a map whose source is a model of the $2$-sphere and whose target is contractible. Hence, we get the following result.
\begin{lemma}\label{lem:non-existence_result_mod_str}
Let $sSet$ have the standard model structure due to Quillen. There is no model structure on $nsSet$ such that $DSd$ is a left Quillen functor.
\end{lemma}
\begin{proof}
Any simplicial set is cofibrant in the standard model structure on $sSet$ due to Quillen. This is because the cofibrations are precisely the degreewise injective maps. See Proposition 3.2.2. in Hovey's book \cite{Ho99} for a reference.

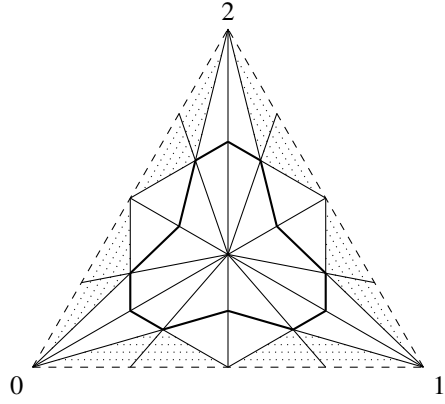
\begin{figure}
\centering
\begin{tikzpicture}

% Vertices of a triangle
\coordinate (2) at (90:3cm);
\coordinate (0) at (210:3cm);
\coordinate (1) at (-30:3cm);

% Nodes to mark vertices of a triangle
\node [above] at (2) {2};
\node [below left] at (0) {0};
\node [below right] at (1) {1};

% Draw line between the vertices 0,1 and 2 and name the midpoints
\draw [dashed] (0.north east)--(2.south) coordinate[midway](02);
\draw [dashed] (0.north east)--(1.north west) coordinate[midway](01);
\draw [dashed] (1.north west)--(2.south) coordinate[midway](12);

% Barycenter, could also be found by means of the intersection library of tikz
\coordinate (012) at (barycentric cs:0=1,1=1,2=1);

% Draws lines from vertices of original triangle to barycentre of original triangle and...
\draw (0.north east)--(012) coordinate[midway](0<012);
\draw (1.north west)--(012) coordinate[midway](1<012);
\draw (2.south)--(012) coordinate[midway](2<012);
\draw (01)--(012) coordinate[midway](01<012);
\draw (12)--(012) coordinate[midway](12<012);
\draw (02)--(012) coordinate[midway](02<012);

% Subdivide each of the six new 2-simplices. To highlight the structure of the desingularized simplicial set we make certain lines thicker
\coordinate (0<01<012) at (barycentric cs:0=1,01=1,012=1);
\coordinate (0<01) at (barycentric cs:0=0.5,01=0.5);
\foreach \x in {(0),(0<01),(01),(012)}
	\draw (0<01<012)--\x;
\foreach \x in {(0<012),(01<012)}
	\draw [thick] (0<01<012)--\x;

\coordinate (1<01<012) at (barycentric cs:1=1,01=1,012=1);
\coordinate (1<01) at (barycentric cs:1=0.5,01=0.5);
\foreach \x in {(1),(1<01),(01),(012)}
	\draw (1<01<012)--\x;
\foreach \x in {(01<012),(1<012)}
	\draw [thick] (1<01<012)--\x;

\coordinate (1<12<012) at (barycentric cs:1=1,12=1,012=1);
\coordinate (1<12) at (barycentric cs:1=0.5,12=0.5);
\foreach \x in {(1),(1<12),(12),(012)}
	\draw (1<12<012)--\x;
\foreach \x in {(12<012),(1<012)}
	\draw [thick] (1<12<012)--\x;

\coordinate (2<12<012) at (barycentric cs:2=1,12=1,012=1);
\coordinate (2<12) at (barycentric cs:2=0.5,12=0.5);
\foreach \x in {(2),(2<12),(12),(012)}
	\draw (2<12<012)--\x;
\foreach \x in {(12<012),(2<012)}
	\draw [thick] (2<12<012)--\x;
	
\coordinate (2<02<012) at (barycentric cs:2=1,02=1,012=1);
\coordinate (2<02) at (barycentric cs:2=0.5,02=0.5);
\foreach \x in {(2),(2<02),(02),(012)}
	\draw (2<02<012)--\x;
\foreach \x in {(02<012),(2<012)}
	\draw [thick] (2<02<012)--\x;
	
\coordinate (0<02<012) at (barycentric cs:0=1,02=1,012=1);
\coordinate (0<02) at (barycentric cs:0=0.5,02=0.5);
\foreach \x in {(0),(0<02),(02),(012)}
	\draw (0<02<012)--\x;
\foreach \x in {(02<012),(0<012)}
	\draw [thick] (0<02<012)--\x;

% Effect of desingularizing doubly subdivided 2-simplex with collapsed boundary. We make these lines dotted.
\foreach \x in {0.2,0.4,0.6,0.8}
	\draw [dotted] (barycentric cs:0=\x,0<01<012=1-\x)--(barycentric cs:01=\x,0<01<012=1-\x);

\foreach \x in {0.2,0.4,0.6,0.8}
	\draw [dotted] (barycentric cs:1=\x,1<01<012=1-\x)--(barycentric cs:01=\x,1<01<012=1-\x);

\foreach \x in {0.2,0.4,0.6,0.8}
	\draw [dotted] (barycentric cs:1=\x,1<12<012=1-\x)--(barycentric cs:12=\x,1<12<012=1-\x);

\foreach \x in {0.2,0.4,0.6,0.8}
	\draw [dotted] (barycentric cs:2=\x,2<12<012=1-\x)--(barycentric cs:12=\x,2<12<012=1-\x);

\foreach \x in {0.2,0.4,0.6,0.8}
	\draw [dotted] (barycentric cs:2=\x,2<02<012=1-\x)--(barycentric cs:02=\x,2<02<012=1-\x);

\foreach \x in {0.2,0.4,0.6,0.8}
	\draw [dotted] (barycentric cs:0=\x,0<02<012=1-\x)--(barycentric cs:02=\x,0<02<012=1-\x);
\end{tikzpicture}
\caption{Desingularizing the double subdivision of the standard $2$-simplex with collapsed boundary.}
\label{fig:ch1_Desing_doublysubd_2simpl}
\end{figure}

By Ken Brown's lemma \cite[Lem.~1.1.12, p.~6]{Ho99} a left Quillen functor takes each weak equivalence between cofibrant objects to a weak equivalence. However, $DSd(d_{\Delta [2]/\partial \Delta [2]})$ is not a weak equivalence. Thus $DSd$ is not a left Quillen functor.
\end{proof}
\noindent Moreover, the diagram
\begin{displaymath}
\xymatrix{
DSd^2(\Delta [2]) \ar[d]^\sim & DSd^2(\partial \Delta [2]) \ar[l] \ar[d]^\sim \ar[r] & DSd^2(\Delta [0]) \ar[d]^\sim \\
DSd(\Delta [2]) & DSd(\partial \Delta [2]) \ar[l] \ar[r] & DSd(\Delta [0])
}
\end{displaymath}
indicates that the map $DSd(\partial \Delta [2])\to DSd(\Delta [2])$ is most likely a non-candidate for a cofibration whenever $nsSet$ is a left proper model category. \cref{lem:non-candidate_cofibration} below justifies this educated guess.

We recall the axiom of propriety, which is desirable in a model category. Consider a commutative square
\begin{displaymath}
 \xymatrix{
 X \ar[d]_i \ar[r]^f & Z \ar[d]^j \\
 Y \ar[r]_g & W
 }
\end{displaymath}
in some category. If the square is cartesian, then we say that $f$ is the \textbf{base change of $g$ along $j$}. If it is cocartesian, then we say that $g$ is the \textbf{cobase change of $f$ along $i$}.
\begin{definition}
Consider a model category. We say that the model category is \textbf{right proper} if weak equivalences are preserved under taking base change along fibrations. Consider a model category. We say that the model category is \textbf{left proper} if weak equivalences are preserved under taking cobase change along cofibrations. If a model category is both right proper and left proper, then we say that it is \textbf{proper}.
\end{definition}
\noindent Note that $sSet$ with the standard model structure is proper \cite[Thm.~13.1.13, p.~242]{Hi03}.

There is a gluing lemma that says that if we have a commutative diagram
\begin{displaymath}
\xymatrix{
B \ar[d]^\sim & A \ar[l] \ar[d]^\sim \ar[r] & C \ar[d]^\sim \\
Y & X \ar[l] \ar[r] & Z
}
\end{displaymath}
in a left proper model category such that at least one map in each row is a cofibration and such that all the vertical maps are weak equivalences, then the canonical map
\[B\sqcup _AC\xrightarrow{\sim } Y\sqcup _XZ\]
of pushouts is a weak equivalence. A reference for the dual of this result is Proposition 13.3.9 in Hirschhorn's book \cite[pp.~246--247]{Hi03}. Note that a more common glueing lemma demands that $A\to B$ and $X\to Y$ be cofibrations and not simply that at least one map in each row be a cofibration.

The former of the two versions of the glueing lemma yields the following result.
\begin{lemma}\label{lem:non-candidate_cofibration}
Assume that $nsSet$ is given a model structure such that it is a left proper model category whose weak equivalences are those maps $f$ such that $\lvert Uf\rvert$ is a (weak) homotopy equivalence. Then neither of the two maps
\[DSd(\partial \Delta [2])\to DSd(\Delta [2])\]
and
\[DSd(\partial \Delta [2])\to DSd(\Delta [0])\]
is a cofibration or neither of the two maps
\[DSd^2(\partial \Delta [2])\to DSd^2(\Delta [2])\]
and
\[DSd^2(\partial \Delta [2])\to DSd^2(\Delta [0])\]
is a cofibration.
\end{lemma}
\noindent \cref{lem:non-candidate_cofibration} justifies the educated guess that $DSd(\partial \Delta [2])\to DSd(\Delta [2])$ is most likely not a cofibration.

Before we can state the nature of these properties we need a few definitions. Let $\varepsilon ^n_j:[0]\to [n]$ be the \textbf{vertex operator} given by $0\mapsto j$. Usually, we omit the upper index.
\begin{definition}
Let $X$ be a simplicial set, and $A$ a simplicial subset. We say that $A$ is \textbf{full} if it has the property that any simplex of $X$ is a simplex of $A$ provided its vertices
are in $A$.
\end{definition}
\begin{definition}
Suppose $X$ a simplicial set. Let $A$ be a full simplicial subset of $X$. We say that $A$ is an \textbf{eden (resp. abyss)} in $X$ if it has the property that any $1$-simplex $x$ of $X$ whose first (resp. zeroth) vertex $x\varepsilon _1$ (resp. $x\varepsilon _0$) is in $A$, is itself is a simplex of $A$.
\end{definition}
\noindent We wish to compare the new notions with analogous notions in the $Cat$, partly because the intuition is more readily available in $Cat$ than in $sSet$.

Consider the notions of sieve and cosieve.
\begin{definition}
Suppose $\mathscr{C}$ a small category. Let $\mathscr{D}$ be a subcategory of $\mathscr{C}$. We will say that $\mathscr{D}$ is a \textbf{(co)sieve} in $\mathscr{C}$ if whenever we have a morphism $c\to c'$ whose target (source) is an object of $\mathscr{D}$, then the morphism is itself a morphism of $\mathscr{D}$.
\end{definition}
\noindent Intuitively, a sieve is a place to which there is no entry and a cosieve is a place from which there is no escape. The notion of sieve corresponds to the notion of eden and the notion of cosieve corresponds to the notion of abyss. In $PoSet$, the notion of sieve is equivalent to the notion of ideal when a poset is thought of as a set equipped with a reflexive, antisymmetric and transitive binary operation.

Note the following relationship between the notions of sieve and eden and between cosieve and abyss.
\begin{lemma}\label{lem:nerve_of_(co)sieve}
The nerve of a sieve (resp. cosieve) is an eden (resp. abyss).
\end{lemma}
\noindent Furthermore, note the following characterization.
\begin{lemma}\label{lem:(co)sieve_characterization_elementary}
A simplicial subset $A$ of a simplicial set $X$ is an eden in $X$ if and only if any simplex whose last vertex is in $A$ is also a simplex of $A$. Similarly, the simplicial subset $A$ is an abyss in $X$ if and only if any simplex whose zeroth vertex is in $A$ is also a simplex of $A$.
\end{lemma}
\noindent \cref{lem:(co)sieve_characterization_categorical} below provides another characterization that is more useful.

Performing desingularization is messy in general. However, there are useful situations in which the process is predictable. Such as when one desingularizes a quotient $X/A$ of a non-singular simplicial set $X$ by an eden $A$. \cref{prop:desingularizing_after_collapsing_elysium} will make this precise. Understanding the behavior of $D$ towards quotients of the kind we mentioned is vital to our discussion of the properties of Str\o m maps.

The new notions are of the following categorical nature.
\begin{lemma}\label{lem:(co)sieve_characterization_categorical}
A simplicial subset $A$ of a simplicial set $X$ is an eden (resp. abyss) if and only if there is a map
$\chi :X\rightarrow \Delta [1]$ such that the square
\begin{displaymath}
 \xymatrix{
 A \ar[d] \ar[r] & \Delta [0] \ar[d]^{N\varepsilon _0\; (\textrm{resp.} \; N\varepsilon _1)} \\
 X \ar[r]_(.45)\chi & \Delta [1]
 }
\end{displaymath}
is cartesian. Here,
\[\varepsilon _0:[0]\to [1]\; (\textrm{resp.} \; \varepsilon _1:[0]\to [1])\]
is the vertex operator given by
\[0\mapsto 0\; (\textrm{resp.} \; 0\mapsto 1).\]
We refer to $\chi$ as the \textbf{characteristic map of $A$ as an eden (resp. abyss) in $B$}.
\end{lemma}
\noindent The proof of this lemma is straight-forward, and is left out.

Part of the interest in the notion of eden is that the Kan subdivision creates edens from arbitrary simplicial subsets, which we state as \cref{lemma_subdivision_creates_left_sieve} below. First, we remind the reader how to define the Kan subdivision.

Consider a simplicial set $X$ and the poset $X^\sharp$ of non-degenerate simplices. There is a morphism $y\to x$ from $y$ to $x$ if $y$ is a face of $x$. The operation of taking a simplicial set $X$ to $X^\sharp$ defines a functor $(-)^\sharp :sSet\to PoSet$.  A map $f:X\to Y$ induces the map $f^\sharp :X^\sharp \to Y^\sharp$ given by sending $x$ to the non-degenerate part $f(x)^\sharp$ of $f(x)$.
\begin{lemma}\label{lem:sharp_creates_sieves}
Let $X$ be a simplicial set and let $A$ be a simplicial subset of $X$. Then $A^\sharp$ is a sieve in $X^\sharp$.
\end{lemma}
\noindent This observation will be used in the proof of \cref{lem:Barratt_nerve_of_inclusion_of_non-sing_is_strom} below.
\begin{definition}\label{def:Barratt_nerve}
We refer to the endofunctor of simplicial sets defined on objects by $BX=N(X^\sharp )$ as the \textbf{Barratt nerve}.
\end{definition}
\noindent Note that this terminology is not standard. We follow \cite[Def.~2.2.3, p.~35]{WJR13}, but Fritsch and Piccinini call $B$ the \emph{star functor} \cite[Exercise~4.6.33,~p.~219]{FP90}. The \textbf{Kan subdivision} is the left Kan extension of $B$ along the Yoneda embedding $\Upsilon :\Delta \to sSet$. Loosely, the Kan subdivision is the best way to adapt \emph{barycentric subdivision} to simplicial sets.

We can elaborate the previous paragraph. The \textbf{simplex category} of $X$, denoted $\Delta \downarrow X$, is the small category whose objects are the representing maps $\bar{x}$ of simplices of $X$ and whose morphisms $\bar{y} \to \bar{x}$ are the commutative diagrams
\begin{displaymath}
\xymatrix@=1em{
\Delta [m] \ar[dr]_{\bar{y} } \ar[rr]^{\alpha } && \Delta [n] \ar[ld]^{\bar{x} } \\
& X
}
\end{displaymath}
whenever $y$ is of degree $m$ and $x$ is of degree $n$. Note that we simplify the notation slightly by writing $\alpha$ in place of $N\alpha$, where $\alpha :[m]\to [n]$ must by definition be an operator such that $y=x\alpha$.

One can view the Kan subdivision of $X$ as
\[Sd\; X\cong colim(B\circ \Upsilon _X),\]
where $\Upsilon _X:\Delta \downarrow X\to sSet$ is the composite of Yoneda embedding $[n]\xmapsto{\Upsilon } \Delta [n]$ with the forgetful functor $(x,n)\mapsto [n]$. A simplicial map $f:X\to Y$ gives rise to a functor $\Delta \downarrow f$ such that $\Upsilon _X=\Upsilon _Y\circ \Delta \downarrow f$. In particular, the identity is a natural transformation
\[\Upsilon _X\Rightarrow\Upsilon _Y\circ \Delta \downarrow f.\]
From this arises the map $Sd(f):Sd\, X\to Sd\, Y$ in an intuitive way.

Combining the diagram $B\circ \Upsilon _Y$ with its colimit $Sd\, Y$ gives rise to a cocone on $B\circ \Upsilon _Y\circ \Delta \downarrow f$ with apex $Sd\, Y$ and thus a map
\[colim(B\circ \Upsilon _Y\circ \Delta \downarrow f)\to Sd\, Y.\]
The identity natural transformation $\Upsilon _X\Rightarrow\Upsilon _Y\circ \Delta \downarrow f$ gives rise to a natural transformation
\[B\circ \Upsilon _X\Rightarrow B\circ \Upsilon _Y\circ \Delta \downarrow f,\]
which must be the identity as well. Thus the map above with target $Sd\, Y$ can be considered to have $Sd\, X$ as its source. The map itself is denoted $Sd(f)$.

We can take the viewpoint that
\[X\cong colim(\Upsilon X)\]
\cite[Lem.~4.2.1~(ii),~p.~141]{FP90}. In other words, the cocone $\Upsilon _X\Rightarrow \underline{X}$, meaning the natural transformation from $\Upsilon _X$ to the constant diagram that takes every object to $X$, is universal. Combining this with $B$ yields a cocone $B\circ \Upsilon _X\Rightarrow \underline{BX}$ with apex $BX$. It gives rise to a canonical map $b_X:Sd\, X\to BX$.
\begin{lemma}
\label{lem:properties_of_b_X}
The canonical map $b_X:Sd\, X\to BX$ is natural, degreewise surjective and an isomorphism if and only if $X$ is non-singular.
\end{lemma}
\begin{proof}
The naturality is automatic when $b_X$ comes from the viewpoint that $Sd$ is the left Kan extension of $B$ along the Yoneda embedding. See \cite[Lem.~2.2.10, p.~38]{WJR13} for the statement and proof that $b_X$ is degreewise surjective. See \cite[Lem.~2.2.11, p.~38]{WJR13} for the statement and proof that $b_X$ is an isomorphism if and only if $X$ is non-singular.
\end{proof}
\noindent We will make use of the comparison map $b_X$ in the proof of the crucial result stated as \cref{cor:two-fold_subdivision_strom}.

As promised, the Kan subdivision creates sieves.
\begin{lemma}\label{lemma_subdivision_creates_left_sieve}
Let $X$ be a simplicial set and $A$ a simplicial subset. Then $\textrm{Sd} \, A$ is an eden in $\textrm{Sd} \, X$.
\end{lemma}
\begin{proof}[Proof of Lemma~\ref{lemma_subdivision_creates_left_sieve}]
Let $i:A\to X$ be the inclusion. We will construct a natural transformation
\[B\circ \Upsilon _X\xRightarrow{\psi } \underline{\Delta [1]},\]
which gives rise to a map $\chi :\textrm{Sd} \,X\to \Delta [1]$. Next, we will verify that $Sd\, A\to \Delta [0]$ is a base change of $\chi$ along $N\varepsilon _0$.

Given an object $\bar{x} :\Delta [n]\to X$ of $\Delta \downarrow X$ we define
\[\psi _{\bar{x} }:B(\Delta [n])\to \Delta [1]\]
by letting it be the nerve of $\Delta [n]^\sharp \to [1]$ given by sending an object $\mu$ of $\Delta [n]^\sharp$ to $0$ if $x\mu$ is a simplex of $A$, and to $1$ otherwise.

We verify that the triangle
\begin{displaymath}
\xymatrix@=1em{
 B(\Delta [m]) \ar[dd]_{B(N\alpha )} \ar[dr]^{\psi _{\bar{y} }} \\
 & \Delta [1] \\
 B(\Delta [n]) \ar[ur]_{\psi _{\bar{x} }}
 }
\end{displaymath}
commutes whenever $\alpha$ is such that $y=x\alpha$. To this end, take some face operator $\mu \in \Delta [m]^\sharp$ with target $[m]$. The order-preserving function $(N\alpha )^\sharp$ sends $\mu$ to the face operator $(\alpha \mu )^\sharp$. We can write $y\mu$ as a degeneracy
\[y\mu =x\alpha \mu =x(\alpha \mu )^\sharp (\alpha \mu )^\flat\]
of $x(\alpha \mu )^\sharp$. This means that $y\mu$ is a simplex of $A$ if and only if $x(\alpha \mu )^\sharp$ is a simplex of $A$. In other words, the underlying triangle of posets commutes. Thus $\psi _{\bar{x} }$ is natural, as claimed.

As a result of the previous paragraph we now have the composite natural transformation
\[B\circ \Upsilon _X\Rightarrow \underline{Sd\, X} \Rightarrow \underline{\Delta [1]} \]
between functors $\Delta \downarrow X\to sSet$. This composite induces a composite of natural transformations between functors $\Delta \downarrow A\to sSet$, through precomposition with $\Delta \downarrow i$. By the design of $\psi$, the latter factors through $N\varepsilon _0:\Delta [0]\to \Delta [1]$. This way we obtain a commutative square
\begin{displaymath}
\xymatrix{
B\circ \Upsilon X\circ \Delta \downarrow i \ar@{=>}[d] \ar@{==>}[r] & \underline{\Delta [0]} \ar@{=>}[d]^{\underline{N\varepsilon _0} } \\
\underline{Sd\, X} \ar@{=>}[r] & \underline{\Delta [1]}
}
\end{displaymath}
of natural transformations and thus a candidate $\chi :Sd\, X\to \Delta [1]$ for a characteristic map. It remains to verify that, if given a solid arrow commutative diagram
\begin{displaymath}
 \xymatrix{
 Z \ar@/^2pc/[drr] \ar@{-->}[dr] \ar@/_2pc/[ddr]_f \\
 & Sd\, A \ar[d]_{Sd(i)} \ar[r] & \Delta [0] \ar[d]^{N\varepsilon _0} \\
 & Sd\, X \ar[r]_(.45)\chi & \Delta [1]
 }
\end{displaymath}
then there exists a dashed map $Z\to Sd\, A$ that makes the whole diagram commute. There is at most one such map $Z\to Sd\, A$ as $Sd(i)$ is degreewise injective. Because $\Delta [0]$ is a terminal object it is enough to verify that $f$ factors through $Sd(i)$. As $Sd(i)$ is degreewise injective it suffices to verify that the image of $f$ is contained in the image of $Sd(i)$.

Suppose $z$ a $q$-simplex of $Z$. By the commutativity of the solid arrow diagram, we get that
\[N\varepsilon _0 \circ g(z)=\chi \circ f(z).\]
We argue that $f(z)\in Sd(X)_q$ is in the image of $Sd(i)_q$.

The simplex $f(z)$ is the image of some element $\varphi :[q]\to \Delta [n]^\sharp \in \Upsilon _X(\bar{x} )$ such that $\varphi (q)$ is the identity. Write $\varphi _j=\varphi (j)$ for $0\leq j\leq q$. Because $\chi \circ f(z)$ is in the image of $N\varepsilon _0$, it follows that $x\varphi _j$ is a simplex of $A$ for each $j$ with $0\leq j\leq q$. In particular, the simplex $x\varphi _q$ is a simplex of $A$. The face operator $\varphi _q$ is the identity, so $x=x\varphi _q$ is itself a simplex of $A$. Thus $f(z)$ is in the image of $Sd(i)$.
\end{proof}
\noindent Now we know that a simplicial subset of a simplicial set can always be turned into an eden by applying the Kan subdivision.

\newpage{}
The following term is central.
\begin{definition}\label{def:strom}
A map $k:A\to B$ in $nsSet$ is referred to as a \textbf{Strøm map} if the following conditions hold.
\begin{enumerate}
\item{The map $k$ is a degreewise injective map whose image is an eden in $B$.}
\item{There is an abyss $W$ in $B$ such that $k$ can be factored as $i:A\to W$ followed by the inclusion $j:W\to B$.}
\item{The map $i$ is a section of some map $r:W\to A$.}
\item{The simplicial set $W$ is deformable rel $A$ to $A$ in $W$, namely there exists a simplicial homotopy $\epsilon :W\times \Delta [1]\to W$ such that the diagrams
\begin{displaymath}
\xymatrix{
W \ar[d]_{i_0} \ar[dr]^{ir} && A\times \Delta [1] \ar[dd]_{pr_1} \ar[r]^{i\times1} & W\times \Delta [1] \ar[dd]_{\epsilon } \\
W\times \Delta [1] \ar[r]^\epsilon & W  \\
W \ar[u]^{i_1} \ar[ur]_1 && A \ar[r]^i & W
}
\end{displaymath}
commute.}
\end{enumerate}
\end{definition}
\noindent Notice that the image of $k$ is an eden in $W$.

The class of Str\o m maps is not a category as a composite of Str\o m maps is not necessarily a Str\o m map. This is because the two simplicial homotopies as described in \cref{def:strom} that come with two composable Str\o m maps do not necessarily give rise to a new simplicial homotopy that satisfies the fourth condition of \cref{def:strom}. Compare with class of pseudo-Dwyer maps \cite{Ci99}, which does form a subcategory of $Cat$.

A bit of history may be of interest to the reader. The class of Str\o m maps fills the same role in establishing $nsSet$ as a model category (Quillen equivalent to $sSet$) as the class of \emph{Dwyer maps} in Thomason's paper \cite{Th80}, where $Cat$ is established as a model category (Quillen equivalent to $sSet$). However, a mistake in Thomason's proof intially left the axiom of propriety unproven.

After having established the model structure, Thomason asserted that the Dwyer maps were closed under retracts. As any cofibration was a retract of a Dwyer map, Thomason concluded that any cofibration was a Dwyer map. Therefore, as the nerve functor $N$ took a cocartesian square in $Cat$ with at least one leg Dwyer to a homotopy cocartesian square in $sSet$, it would follow that $Cat$ is left proper. However, the Dwyer maps are not closed under retracts \cite{Ci99}.

This mistake was not a fatal mistake, as it turned out. Cisinski was able to correct the proof of the axiom of propriety by weakening the definition of the term Dwyer map and thus creating a new notion that he gave the ad hoc name \emph{pseudo-Dwyer map}. Perhaps the new notion is better referred to under the name \emph{Cisinski map}. The notion of Cisinski map may have been borrowed from A. Str\o m as it is an analogue to one of his characterizations \cite[Thm.~2~(ii), p.~12]{St66} of the cofibrations for the Str\o m model structure on topological spaces \cite{St72}. It is the model structure whose weak equivalences are the homotopy equivalences and whose fibrations are the Hurewicz fibrations.

Cisinski argues that $N$ takes a cocartesian square in $Cat$ with at least one leg Cisinski to a homotopy cocartesian square in $sSet$ \cite{Ci99}. Thus Thomason's argument that $Cat$ is left proper goes through when Dwyer maps are replaced by Cisinski maps. Cisinski takes the correction one step further and points out that Cisinski maps are closed under cobase change and under taking compositions of $\aleph _0$-sequences \cite{Ci99}. Indeed, Raptis points out that both Dwyer maps and Cisinski maps are closed under (transfinite) compositions \cite[Prop.~2.4.~(a),~p.~216]{Ra10}. Thus, using Thomason's original technique, Thomason's model structure on $Cat$ can be established by means of the term Cisinski map alone, although the notion of Dwyer map plays a role in Thomason's discussion regarding cofibrant objects \cite[Lemma 5.6.~(4),p.~323]{Th80}.

Crucially, the sets $DSd^2(I)$ and $DSd^2(J)$ are contained in the class of Str\o m maps, as we will now argue.
\begin{lemma}\label{lem:Barratt_nerve_of_inclusion_of_non-sing_is_strom}
Let $k:A\to X$ be an inclusion of a simplicial subset $A$ into a non-singular simplicial set $X$. If $A$ is an eden in $X$, then $B(k)$ is a Str\o m map.
\end{lemma}
\begin{proof}
Let $W$ be the subposet of $X^\sharp$ whose objects are precisely the non-degenerate simplices of $X$ that have a face in $A$. As $A$ is an eden it follows that there is a greatest face in $A$ of any given element of $W$. If $w\in W$, we let $r(w)$ denote this unique face. Because $X$ is non-singular it follows that $r(w)$ is non-degenerate, hence an object of $A^\sharp$. Moreover, we get a functor $r:W\to A^\sharp$. It is a retraction of the corestriction $i$ of $k^\sharp :A^\sharp \to X^\sharp$ to $W$.

By \cref{lem:sharp_creates_sieves}, the functor $(-)^\sharp$ creates sieves. Therefore, we get that $A^\sharp$ is a sieve in $X^\sharp$. By the definition of $W$ it follows that it is a cosieve in $X^\sharp$. Furthermore, \cref{lem:nerve_of_(co)sieve} says that $BA=N(A^\sharp )$ is an eden in $BX=N(X^\sharp )$ and that $NW$ is an abyss in $BX$.

If $w\in W$, then there is a morphism $ir(w)\to w$ by the definition of $r$. The rest of the argument is standard. Namely, because $W$ is a poset it is true that $ir(w)\to w$ is automatically natural. This natural morphism from $ir$ to the identity can be viewed as a functor $W\times [1]\to W$, which in turn gives rise to a simplicial homotopy $NW\times \Delta [1]\to NW$ from $Ni\circ Nr$ to the identity as $N$ preserves limits and in particular products. The simplicial homotopy is stationary on $N(A^\sharp )$ because it is identified with the nerve of $W\times [1]\to W$, which is stationary on $A^\sharp$ in an intuitive, analogous sense. This concludes the proof that $B(k)$ is a Str\o m map.
\end{proof}
\begin{corollary}\label{cor:two-fold_subdivision_strom}
Let $Y$ be a simplicial set such that $Sd\, Y$ is non-singular and let $X$ be a simplicial subset of $Y$. If $k:X\to Y$ is the inclusion, then $Sd^2(k)$ is Str\o m.
\end{corollary}
\begin{proof}
According to \cref{lemma_subdivision_creates_left_sieve}, we have that $Sd\, X$ is an eden in $Sd\, Y$. By \cref{lem:Barratt_nerve_of_inclusion_of_non-sing_is_strom} we now know that $BSd(k)$ is Str\o m. The naturality of $b_{Sd\, X}$ means that we can identify $BSd(k)$ with $Sd^2(k)$ via the diagram
\begin{displaymath}
\xymatrix{
Sd^2\, X \ar[d]_{Sd(Sd(k))} \ar[r]^{b_{Sd\, X}}_\cong & BSd\, X \ar[d]^{B(Sd(k))} \\
Sd^2\, Y \ar[r]_{b_{Sd\, Y}}^\cong & BSd\, Y
}
\end{displaymath}
as $Sd\, X$ and $Sd\, Y$ are non-singular. This is because the natural map from the Kan subdivision to the Barratt nerve is an isomorphism when the original simplicial set is non-singular \cite[Lem.~2.2.11, p.~38]{WJR13}. Hence, the map $Sd^2(k)$ is a Str\o m map.
\end{proof}
\noindent In particular, if $k$ in \cref{cor:two-fold_subdivision_strom} is one of the inclusions $\partial \Delta [n]\to \Delta [n]$ or one of the inclusions $\Lambda ^j[n]\to \Delta [n]$, then we see that $Sd^2(k)$ is a Str\o m map.

%% file: sections/planes.tex
\section{On higher and lower planes of existence}
\label{sec:planes}

\noindent \cref{cor:two-fold_subdivision_strom} has shown us that the sets $DSd^2(I)$ and $DSd^2(J)$ are both contained in the class of Str\o m maps. This class of maps will serve as an auxiliary class of maps that aids us in establishing the model structure.

To form a pushout in $nsSet$ one can first form the pushout in $sSet$ and then desingularize it. The desingularization process destroys the homotopy type in general, but it turns out that the homotopy type is preserved when the pushout in $sSet$ is taken along a Str\o m map. This result is stated as \cref{lem:Pushout_along_strom_homotopically_wellbehaved}. The important formal property of Str\o m maps is that they are preserved under taking cobase change, which is stated as \cref{prop:Strom-maps_closed_under_cobasechange}. To prove both of these results, the most work intensive task is to establish \cref{prop:desingularizing_after_collapsing_elysium}, which we will focus on in this section. It helps us control the homotopical behavior of desingularization in important cases.

As a preliminary step towards proving that Str\o m maps are preserved under cobase change, we have the following basic result.
\begin{lemma}\label{lem:elysiums_abysses_preserved_cobase_change}
If the square
\begin{displaymath}
\xymatrix{
A \ar[d]_i \ar[r]^f & C \ar[d]^j \\
X \ar[r]_(.35)g & X\sqcup _AC
}
\end{displaymath}
is cocartesian in $sSet$ and $i$ embeds $A$ as an eden (resp. abyss) in $X$ then $j$ embeds $C$ as an eden (resp. abyss) in $X\sqcup _AC$.
\end{lemma}
\begin{proof}
We do the case when $A$ is an eden. Notice that no part of the proof prefers the case when $A$ is an eden over the case when $A$ is an abyss. Alternatively, use the notion of the opposite \cite[Def.~2.2.19, p.~ 42]{WJR13} of a simplial set to conclude that the result also holds in the case when $A$ is an abyss.

Note that we can factor $f:A\to C$ as a degreewise surjective map followed by a degreewise injective map, so we can prove the lemma by proving that it holds in the two cases when $f$ is degreewise surjective or degreewise injective.

First, we do the case when $f$ is degreewise surjective. Suppose $y$ some simplex of $X\sqcup _AC$ whose last vertex is in the image of $j$. We will prove that $y$ is in the image of $g$. Here, we use the elementary characterization from \cref{lem:(co)sieve_characterization_elementary}.

There is at most one simplex $x$ such that $y=g(x)$. Suppose there is one. As $f$ is surjective in degree $0$, there is a $0$-simplex $v$ of $A$ such that
\[y\varepsilon _n=j\circ f(v)=g\circ i(v)\]
by the assumption that $y\varepsilon _n$ is in the image of $j$. As $i$ embeds $A$ as an eden in $X$, there is a simplex $a$ of $A$ such that $x=i(a)$. Then we can define $c=f(a)$. The given simplex $y$ is the image under $j$ of $c$. It follows that $j$ embeds $C$ as an eden in $X\sqcup _AC$.

Finally, we do the case when $f$ is degreewise injective. Suppose $y$ some simplex of $X\sqcup _AC$ whose last vertex is in the image of $j$. We will prove that $y$ is in the image of $g$.

There is at most one simplex $x$ such that $y=g(x)$. Suppose there is one. The vertex $y\varepsilon _n$ is then uniquely the image under $g$ of $x\varepsilon _n$, in addition to being uniquely the image under $j$ of some $0$-simplex $w$ of $C$. Hence, there is some unique $0$-simplex $v$ of $A$ whose images under $f$ and $i$ are $w$ and $x\varepsilon _n$, respectively. Hence, there is some simplex $a$ of $A$ with $x=i(a)$ by the assumption that $i$ embeds $A$ as an eden in $X$. Thus $y$ is the image under $j$ of $c=f(a)$. It follows that $j$ embeds $C$ as an eden in $X\sqcup _AC$.
\end{proof}
\noindent In addition to \cref{lem:elysiums_abysses_preserved_cobase_change}, we will state some basic properties of cartesian squares.

The properties stated in \cref{Lemma_Pullbacks_close_to_twooutofthree_property} below are here collectively referred to as the two-out-of-three property for cartesian squares. See for example III.4 Exercise 8 (b) in \cite{ML98} for a reference to the first two statements of \cref{Lemma_Pullbacks_close_to_twooutofthree_property} below. All three statements of \cref{Lemma_Pullbacks_close_to_twooutofthree_property} appear in Lemma 2.4 of \cite[p.~57]{CPS06} for the case $\mathscr{C} =sSet$ as Chachólski, Pitsch and Scherer work in that category.
\begin{lemma}[Two-out-of-three property for cartesian squares]
\label{Lemma_Pullbacks_close_to_twooutofthree_property}
Suppose
\begin{displaymath}
\xymatrix{
A \ar[d] \ar[r] & C \ar[d] \ar[r] & E \ar[d] \\
B \ar[r] & D \ar[r] & F
}
\end{displaymath}
a diagram in some category $\mathscr{C}$.
\begin{enumerate}
\item{The outer square is cartesian if both the left hand and the right hand square are cartesian squares.}
\item{Likewise, the left hand square is cartesian if the right hand and outer squares are cartesian.}
\item{If the outer and left hand squares are cartesian, then the right hand square is cartesian if the morphism $B\to D$ has a section.}
\end{enumerate}
\end{lemma}
\begin{proof}
Consider the third statement, meaning the case when the left hand and outer squares are cartesian and $k$ has a section, consider the diagram
\begin{displaymath}
\xymatrix{
& X \ar@{-}@/_0.2pc/[d] \ar@{-->}[dr]^\gamma \ar@/^1pc/[drrr]^\epsilon \\
A \ar[d]_f \ar[rr]^(.3)i & \ar@/_/[dr]^(.4)\delta & C \ar[d]^g \ar[rr]^j && E \ar[d]^h \\
B \ar[rr]^k && D \ar@/^1pc/[ll]^s \ar[rr]_l && F
}
\end{displaymath}
in $\mathscr{C}$, where we assume that $h\circ \epsilon =l\circ \delta$.

We will prove the existence and uniqueness of a map $\gamma :X\to C$ such that $\epsilon =j\circ \gamma$ and $\delta =g\circ \gamma$.

First we prove existence. Because the outer square is cartesian and because $s$ is a section of $k$, the two maps $\epsilon$ and $s\circ \delta$ give rise to a map $\alpha :X\to A$ such that
\begin{equation}\label{eq:one_proof_of_Lemma_Pullbacks_close_to_twooutofthree_property}
\epsilon =(j\circ i)\circ \alpha
\end{equation}
and
\begin{equation}\label{eq:two_proof_of_Lemma_Pullbacks_close_to_twooutofthree_property}
s\circ \delta =f\circ \alpha .
\end{equation}
Define $\gamma =i\circ \alpha$. Then (\ref{eq:one_proof_of_Lemma_Pullbacks_close_to_twooutofthree_property}) is the first half of what we need to verify. For the second half, observe that $k$ composed with each side of (\ref{eq:two_proof_of_Lemma_Pullbacks_close_to_twooutofthree_property}) yields
\begin{displaymath}
\begin{array}{rcl}
\delta & = & (k\circ s)\circ \delta \\
& = & k\circ (s\circ \delta ) \\
& = & k\circ (f\circ \alpha ) \\
& = & (k\circ f)\circ \alpha  \\
& = & (g\circ i)\circ \alpha  \\
& = & g\circ (i\circ \alpha ) \\
& = & g\circ \gamma ,
\end{array}
\end{displaymath}
which is the second half of the verification of the existence of $\gamma$.

Finally, we prove uniqueness of $\gamma$. Take two maps $X\to C$, denoted $\gamma$ and $\gamma '$, such that the equations
\begin{displaymath}
\begin{array}{rcl}
\delta & = & g\circ \gamma \\
\delta & = & g\circ \gamma ' \\
\epsilon & = & j\circ \gamma \\
\epsilon & = & j\circ \gamma '
\end{array}
\end{displaymath}
hold. Then the two maps $s\circ \delta$ and $\gamma$ give rise to a canonical map $\alpha :X\to A$ as the left hand square is cartesian. Similarly, the two maps $s\circ \delta$ and $\gamma '$ give rise to a canonical map $\alpha ':X\to A$. Next, we can take advantage of the assumption that the outer square is cartesian. This shows that $\alpha =\alpha '$. Then the equations
\[\gamma =i\circ \alpha =i\circ \alpha '=\gamma '\]
yield the desired uniqueness.
\end{proof}
\noindent Note that the assumption that $B\to D$ is an epimorphism is enough for the third statement of \cref{Lemma_Pullbacks_close_to_twooutofthree_property} to hold for for some categories $\mathscr{C}$. This is trivially true when $\mathscr{C} =Set$ is the category of sets and functions, for the epimorphisms are in that case the surjective functions, which are in turn the functions that have a section.
\begin{corollary}\label{cor:sSet_Pullbacks_close_to_twooutofthree_property}
Suppose
\begin{displaymath}
\xymatrix{
A \ar[d] \ar[r] & C \ar[d] \ar[r] & E \ar[d] \\
B \ar[r] & D \ar[r] & F
}
\end{displaymath}
a diagram in the category $sSet$. If the outer and left hand squares are cartesian, then the right hand square is cartesian if $B\to D$ is degreewise surjective.
\end{corollary}
\begin{proof}
The corollary follows from the third statement of \cref{Lemma_Pullbacks_close_to_twooutofthree_property} in the following way. The category $sSet$ is the category of functors $\Delta ^{op} \to Set$ and natural transformations between them. As a $Set$-valued functor category, the category $sSet$ is bicomplete. In a functor category, limits and colimits are formed pointwise. In other words, we can apply \cref{Lemma_Pullbacks_close_to_twooutofthree_property} in the case when $\mathscr{C} =Set$, in a given degree $n$ as $B_n\to D_n$ is surjective by assumption. The right hand square in degree $n$ is thus cartesian. We can conclude that the right hand square of the given diagram is cartesian in $sSet$
\end{proof}
\noindent Note that \cref{cor:sSet_Pullbacks_close_to_twooutofthree_property} shows that the assumption that $B\to D$ is an epimorphism is sufficient in the case when $\mathscr{C} =sSet$ in \cref{Lemma_Pullbacks_close_to_twooutofthree_property} above.

We are interested in triples $(X,A,V)$ where $X$ is a simplicial set, where $A$ is a non-singular eden in $X$ and where $V$ is a non-singular abyss in $X$. We are particularly interested in two cases. The first is when $A$ is contained in $V$ as this is part of the definition of the term Str\o m map. Secondly, we are interested in the case when $A_0\cup V_0=X_0$ and $A_0\cap V_0=\emptyset$. In this section, we will only consider the second case, however the first case plays a role in the next section.

Notice that if $\chi :X\to \Delta [1]$ is the characteristic map of $A$ as an eden in $X$, then $\chi$ is actually also the characteristic map of $V$ as an abyss in $X$. This is because we are concerned with the special case when $A_0\cup V_0=X_0$ and $A_0\cap V_0=\emptyset$. Therefore, given an $n$-simplex $x$ of $X$ we can consider the diagram
\begin{equation}\label{eq:diagram_proof_of_prop_desingularizing_after_collapsing_elysium}
\begin{gathered}
\xymatrix{
\Delta [k] \ar[d] \ar@{-->}[r] \ar@/^1pc/[rr] & A \ar[d] \ar[r] & \Delta [0] \ar[d]^{N\varepsilon _0} \\
\Delta [n] \ar[r]^{\bar{x} } & X \ar[r]^\chi & \Delta [1] \\
\Delta [n-k-1] \ar[u] \ar@/_1pc/[rr] \ar@{-->}[r] & V \ar[u] \ar[r] & \Delta [0] \ar[u]_{N\varepsilon _1}
}
\end{gathered}
\end{equation}
where we have taken the base changes of $\chi \circ \bar{x}$ along $N\varepsilon _0$ and $N\varepsilon _1$, respectively. Here, we allow $-1\leq k\leq n$ and use the convention $\Delta [-1]=\emptyset$. The vertex $x\varepsilon _j$ is a simplex of $A$ if $j\leq k$ and a simplex of $V$ if $j>k$. The diagram above also illustrates the intuition from \cref{sec:behavior}, which says that a simplex can leave an eden or enter an abyss, but that a simplex can neither enter an eden nor leave an abyss.

Now, consider the case when $x$ is non-degenerate. If $k=-1$, then $x$ is a simplex of $V$, which means that it is embedded in $V$ as $V$ is non-singular. Then $x$ is also embedded in $X$, of course. If $k=n$, then $x$ is a simplex of $A$, which means that it is embedded as $A$ is non-singular. Taking the contrapositive, we get that $k\neq -1$ and that $k\neq n$ if $x$ is not embedded. In particular, it follows that $n>0$ if $x$ is not embedded. But if $n=1$, then $x$ is embedded in the case when $k=0$. This is because $A_0$ and $V_0$ are disjoint and because the vertex $x\varepsilon _0$ is a $0$-simplex of $A$ and because $x\varepsilon _1$ is a $0$-simplex of $V$. So in fact,
\begin{equation}\label{eq:conditions_proof_of_prop_desingularizing_after_collapsing_elysium}
-1\neq k\neq n>1
\end{equation}
when $x$ is non-degenerate and non-embedded.

For the statement of \cref{prop:desingularizing_after_collapsing_elysium}, note that we intend to replace the triple $(X,A,V)$ with the triple $(X/A,\Delta [0],V)$ where $X$ is non-singular. In other words, we specialize quite a lot.
\begin{proposition}\label{prop:desingularizing_after_collapsing_elysium}
Let $X$ be non-singular and $A$ an eden in $X$. Furthermore, consider the cocartesian square
\begin{displaymath}
\xymatrix{
  A \ar[d]_i \ar[r]^f & \Delta [0] \ar[d]^{\bar{\imath} } \\
  X \ar[r]_(.4){\bar{f} } & X/A
}
\end{displaymath}
in $sSet$. If $V$ is the full simplicial subset of $X$ whose $0$-simplices are the ones that are not in $A$, then the composite
\[V\xrightarrow{j} X\xrightarrow{\bar{f} } X/A\xrightarrow{\eta } D(X/A),\]
denoted $\tilde{\jmath }$, is an embedding of $V$ as an abyss in $D(X/A)$.
\end{proposition}
\noindent Notice that $V$ is an abyss in $X$ as $A$ is an eden. It is even true that $V$ is an abyss in $X/A$. If the latter statement is not clear at this time, it will be early in the proof. Thus the triple $(X/A,\Delta [0],V)$ is indeed a specialization from the previous paragraphs.

Recall from the fact that $nsSet$ is a reflective subcategory of $sSet$ that one can make the square from \cref{prop:desingularizing_after_collapsing_elysium} cocartesian in $nsSet$ by desingularizing the pushout $X/A$. Let $\tilde{\imath }$ denote the composite of the canonical map $X/A\xrightarrow{\eta } D(X/A)$ with $\bar{\imath }$. Let $\bar{\jmath } =\bar{f} \circ j$.

The triple $(X,\Delta [0],V)$ is a form of world order, where the eden $\Delta [0]$ can be thought of as a higher plane of existence and the abyss $V$ as a lower plane. A simplex of $X/A$ is thought of as living in this world in the manner explained by the diagram (\ref{eq:diagram_proof_of_prop_desingularizing_after_collapsing_elysium}) and the conditions of (\ref{eq:conditions_proof_of_prop_desingularizing_after_collapsing_elysium}).

We will make use of the following terminology.
\begin{definition}\label{def:hty_sequence}
If $\lambda$ is an ordinal, then a \textbf{$\lambda$-sequence} in a cocomplete category $\mathscr{C}$ is a cocontinous functor $X:\lambda \to \mathscr{C}$, written as
\begin{displaymath}
\xymatrix{
X^{[0]} \ar[r] & X^{[1]} \ar[r] & \cdots \ar[r] & X^{[\beta ]} \ar[r] & \cdots \; ,
}
\end{displaymath}
$\beta <\lambda$. The canonical map
\[X^{[0]}\to colim_{\beta <\lambda }X^{[\beta ]}\]
is the \textbf{composition} of the $\lambda$-sequence. A \textbf{sequence} is a $\lambda$-sequence for some ordinal $\lambda$.
\end{definition}
\noindent For sequences, we sometimes use the same letters that at other times denote simplicial sets. However, we use the brackets in the notation to avoid confusion with skeleton filtrations. This is because $X^n$, $n\geq 0$, denotes the $n$-skeleton of a simplicial set $X$. Also recall that we have taken $X_n$, $n\geq 0$, to mean the set of $n$-simplices of a simplicial set $X$. Both of the two latter notations are standard.

Next, we prove the proposition.
\begin{proof}[Proof of \cref{prop:desingularizing_after_collapsing_elysium}]
We will desingularize the simplicial set $X/A$ in an iterative manner. Each non-embedded non-degenerate simplex of $X/A$ will be made degenerate.

The method we use is similar to how G. Lewis Jr. makes a $k$-space compactly generated by identifying two points whenever they cannot be separated by open sets \cite[p.~158]{Le78}.

Our method is also a modification of \cite[Thm.~2.1.3.]{Fj20-DN}. Moreover, the simplicial set $X/A$ is quite special as it is formed by collapsing an eden within a non-singular simplicial set. This makes it viable to deal with one non-embedded non-degenerate simplex at a time.

Recall that $V$ is defined as the full simplicial subset of $X$ whose $0$-simplices are the ones that are not in $A$. The canonical map $\bar{\imath }$ is by \cref{lem:elysiums_abysses_preserved_cobase_change} an embedding of $\Delta [0]$ as a eden, which says precisely that the first quadrant of the diagram
\begin{displaymath}
 \xymatrix{
  A \ar[d]_i \ar[r] & \Delta [0] \ar[d]^{\bar{\imath } } \ar[r] & \Delta [0] \ar[d]^{N\varepsilon _0} \\
  X \ar@{->>}[r]^(.4){\bar{f} } & X/A \ar@{-->}[r]^{\bar{\chi } } & \Delta [1] \\
  V \ar[u]^j \ar@{-->}[r]_\cong \ar[ur]_{\bar{\jmath } } & V' \ar[u] \ar[r] & \Delta [0] \ar[u]_{N\varepsilon _1}
 }
\end{displaymath}
is cartesian. This yields the canonical map $\bar{\chi }$. In addition, we have formed the cartesian square in the fourth quadrant, which yields the map $V\to V'$. Next, we will argue that the latter map is an isomorphism.

We start by proving that $V\to V'$ is degreewise surjective. The outer part of the lower half is cartesian and so is the fourth quadrant. By \cref{Lemma_Pullbacks_close_to_twooutofthree_property} it then follows that the third quadrant is also cartesian. Hence, the map $V\to V'$ is a base change of the degreewise surjective map $\bar{f}$. Limits in $sSet$ are computed in each degree, and in the category of sets, a base change of a surjective map is again surjective. We can conclude that $V_q\to V_q'$ is surjective
for each $q\geq 0$.

Next, we argue that $V\to V'$ is degreewise injective. Consider the diagram
\begin{displaymath}
\xymatrix{
 V \ar[d]_j & \emptyset \ar[l] \ar[d] \ar[r] & \Delta [0] \ar[d] \\
 X & A \ar[l]^i \ar[r]_(.45)f & \Delta [0]
 }
\end{displaymath}
which gives rise to a canonical map $V\sqcup \Delta [0]\to X/A$ between pushouts in $SSet$. As $A$ is an eden in $X$ and by the definition of $V$, the images of $i$ and $j$ are disjoint. Hence, the map between pushouts is degreewise injective. In particular, the composite $\bar{\jmath }$ is degreewise injective, implying that $V\to V'$ is. In other words, the canonical map $V\xrightarrow{\cong } V'$ is an isomorphism.

We are ready to begin the iterative desingularization of $X/A$. Let $p^0$ be the canonical degreewise surjective map $X/A\xrightarrow{\eta _{X/A} } D(X/A)$ and write
\[D^{[0]}(X/A)=X/A.\]
Here, we use brackets, because we intend to describe a sequence. This is to make the notation reflect that of  \cref{def:hty_sequence}

Furthermore, write
\begin{displaymath}
\begin{array}{rcl}
i^0 & = & \bar{\imath } \\
j^0 & = & \bar{\imath } \\
\chi ^0 & = & \bar{\chi } .
\end{array}
\end{displaymath}
Assume that we for some ordinal $\gamma >0$ have a $\gamma$-sequence of commutative diagrams
\begin{displaymath}
\xymatrix{
& \Delta [0] \ar[ld]_{\tilde{i} } \ar[d]_{i^\beta } \ar[r] & \Delta [0] \ar[d]^{N\varepsilon _0} \\
D(X/A) & D^{[\beta ]}(X/A) \ar[l]_(.45){p^\beta } \ar[r]^(.6){\chi ^\beta } & \Delta [1] \\
& V \ar[lu]^{\tilde{j} } \ar[u]^{j^\beta } \ar[r] & \Delta [0] \ar[u]_{N\varepsilon _1}
}
\end{displaymath}
for $\beta <\gamma$ where\dots
\begin{enumerate}
\item{\dots the two squares are cartesian, where\dots}
\item{\dots $p^\beta$ is degreewise surjective for each $\beta <\gamma$ and where\dots}
\item{\dots each map $D^{[\alpha ]}(X/A)\xrightarrow{f^{\alpha ,\beta }} D^{[\beta ]}(X/A)$, $0\leq \alpha \leq \beta <\gamma$, is also degreewise surjective.}
\end{enumerate}
By the phrase \emph{$\gamma$-sequence of commutative diagrams} used above we mean a functor from the ordinal $\gamma$ to the category of functors whose source is the category
\begin{displaymath}
\xymatrix{
& 2 \ar[ld] \ar[d] \ar[r] & 1 \ar[d] \\
3 & 6 \ar[l] \ar[r] & 0 \\
& 4 \ar[lu] \ar[u] \ar[r] & 5 \ar[u]
}
\end{displaymath}
and whose target is $sSet$. Thus compatibility of all the maps above is implicit in the hypothesis. We will refer to the commutative diagram with index $\beta$ as the \textbf{$\beta$-th stage} of the (iterative) desingularization process, and even to $D^{[\beta ]}(X/A)$ under the same name.

If a simplicial set is not non-singular, then we say that it is \textbf{singular}. Together with the $\gamma$-sequence, assume that for each ordinal $\beta <\gamma$ such that $D^{[\beta ]}(X/A)$ is singular, we have a simplex $x^\beta$ of $X$ such that $f^{0,\beta }(x^\beta )$ is a non-embedded non-degenerate simplex of $D^{[\beta ]}(X/A)$. Suppose $x^\alpha \neq x^\beta$ whenever $\alpha \neq \beta$. Assume that for each ordinal $\beta$ such that $\beta +1<\gamma$, we have that the simplex $f^{0,\beta +1}(x^\beta )$ of $D^{[\beta +1]}(X/A)$ is degenerate. This data will later be used in proving that the iterative desingularizing process does indeed come to a halt.

If $D^{[\gamma ]}(X/A)$ is singular, then let $x^\gamma$ be a simplex of $X/A$ whose image under $f^{0,\gamma }$ is a non-embedded non-degenerate simplex. Suppose $\beta <\gamma$. Notice that $x^\beta \neq x^\gamma$ as the commutative diagram
\begin{displaymath}
 \xymatrix{
 X/A \ar[dr]_{f^{0,\beta }} \ar[rr]^{f^{0,\gamma }} && D^{[\gamma ]}(X/A) \\
 & D^{[\beta ]}(X/A) \ar[ur]_{f^{\beta ,\gamma }} \ar[rr]_{f^{\beta ,\beta +1}} && D^{[\beta +1]}(X/A) \ar[lu]_{f^{\beta +1,\gamma}}
 }
\end{displaymath}
shows. Namely, we have that
\[f^{\beta ,\gamma }\circ f^{0,\beta }(x^\beta )\]
is degenerate whereas $f^{0,\gamma }(x^\gamma )$ is not. Note that this argument concerns both the case when $\gamma$ is a limit ordinal and the case when $\gamma$ is a successor ordinal. In the latter case, the map $f^{\beta +1,\gamma }$ in the diagram above is potentially the identity, which is ok.

If $\gamma$ is a limit ordinal, then we form the colimit of the $\gamma$-sequence of commutative diagrams. Because colimits in a functor category are computed pointwise \cite[Section~V.3]{ML98}, the colimit is a diagram
\begin{displaymath}
\xymatrix{
& \Delta [0] \ar[ld]_{\tilde{i} } \ar[d]_{i^\gamma } \ar[r] & \Delta [0] \ar[d]^{N\varepsilon _0} \\
D(X/A) & D^{[\gamma ]}(X/A) \ar[l]_(.45){p^\gamma } \ar[r]^(.6){\chi ^\gamma } & \Delta [1] \\
& V \ar[lu]^{\tilde{j} } \ar[u]^{j^\gamma } \ar[r] & \Delta [0] \ar[u]_{N\varepsilon _1}
}
\end{displaymath}
where $D^{[\gamma ]}(X/A)$ is the colimit of the $\gamma$-sequence
\[D^{[0]}(X/A)\xrightarrow{f^{0,1}} \cdots \to D^{[\beta ]}(X/A)\xrightarrow{f^{\beta ,\beta +1}} \cdots\]
where $0\leq \beta$ and $\beta +1<\gamma$. Because the colimit of commutative diagrams is filtered, both of the squares are cartesian as filtered colimits commute with finite limits \cite[Section~IX.2]{ML98}. The canonical map $p^\gamma$ is automatically degreewise surjective as each map $p^\beta$, $\beta <\gamma$, is degreewise surjective. Also it follows that $f^{\alpha ,\gamma}$ is degreewise surjective for $\alpha <\gamma$.

Now comes the real work. That is, we look at the case when $\gamma =\beta +1$ is a successor ordinal. If $D^{[\beta ]}(X/A)$ is non-singular, then we simply copy the $\beta$-th stage and give the copy the index $\beta +1$. The map to the latter from the $\beta$-th diagram then consists of identities. Otherwise, if $D^{[\beta ]}(X/A)$ is singular, then write $y=f^{0,\beta }(x^\beta )$. Assume that $y$ is of degree $n$. Note that we are about to make $y$ degenerate and that $\beta$ may be a limit ordinal. So the following text both finishes the limit ordinal case and takes care of the successor ordinal case of our iteration.

We can take the base change of $\chi ^\beta \circ \bar{y}$ along $N\varepsilon _0$ and $N\varepsilon _1$, respectively, and get the diagram
\begin{equation}
\label{eq:diagram_proof_of_prop_desingularizing_after_collapsing_elysium_special}
\begin{gathered}
\xymatrix{
\Delta [k] \ar[d] \ar@{-->}[r] \ar@/^1pc/[rr] & \Delta [0] \ar[d]_{i^\beta } \ar[r] & \Delta [0] \ar[d]^{N\varepsilon _0} \\
\Delta [n] \ar[r]^(.4){\bar{y} } & D^{[\beta ]}(X/A) \ar[r]^(.6){\chi ^\beta } & \Delta [1] \\
\Delta [n-k-1] \ar[u] \ar@/_1pc/[rr] \ar@{-->}[r] & V \ar[u]^{j^\beta } \ar[r] & \Delta [0] \ar[u]_{N\varepsilon _1}
}
\end{gathered}
\end{equation}
similar to (\ref{eq:diagram_proof_of_prop_desingularizing_after_collapsing_elysium}) with the conditions of (\ref{eq:conditions_proof_of_prop_desingularizing_after_collapsing_elysium}). Thus the vertices $y\varepsilon _0$, $\dots$, $y\varepsilon _k$ are in the image of $i^\beta$ and the vertices $y\varepsilon _{k+1}$, $\dots$, $y\varepsilon _n$ are in the image of $j^\beta$.

Because the source of $i^\beta$ is $\Delta [0]$, we have
\[y\varepsilon _0=\cdots =y\varepsilon _k.\]
This means that the simplex $p^\beta (y)$ of $D(X/A)$ can be written $p^\beta (y)=w\rho$, where $\rho :[n]\to [n-k]$ is the degeneracy operator given by $0,\dots ,k\mapsto 0$. Therefore, to make $y$ degenerate by pushing out along $\rho$ is be a step towards desingularizing $D^{[\beta ]}(X/A)$. We will shortly argue that this step is non-trivial, meaning that $k>0$. In fact, the step is optimal.

Note that the composite
\[\Delta [n]\xrightarrow{\bar{y} } D^{[\beta ]}(X/A)\xrightarrow{\chi ^\beta } \Delta [1]\]
is induced by the operator $[n]\to [1]$ given by
\[0,\dots ,k\mapsto 0\]
and
\[k+1,\dots ,n\mapsto 1.\]
This operator can be factored as $\sigma \circ \rho$ where $\sigma :[n-k]\to [1]$ is given by $0\mapsto 0$ and sending all elements greater than $0$ to $1$.

The remarks of the two previous paragraphs give rise to the $(\beta +1)$-th stage. Consider the diagram
\begin{displaymath}
\xymatrix{
\Delta [n] \ar[dd]_{\bar{y} } \ar[rr]^\rho && \Delta [n-k] \ar[ld] \ar[dd]^{\bar{z} } \ar@/^2pc/[ddddrr]^\sigma \\
& D(X/A) \\
D^{[\beta ]}(X/A) \ar[ur]^{p^\beta } \ar@/_2pc/[ddrrrr]_{\chi ^\beta } \ar[rr]_{f^{\beta ,\beta +1}} && D^{[\beta +1]}(X/A) \ar@{-->}[lu]_{p^{\beta +1}} \ar@{-->}[ddrr]^{\chi ^{\beta +1}} \\
\\
&&&& \Delta [1]
}
\end{displaymath}
where we have formed a cobase change
\[f^{\beta ,\beta +1}:D^{[\beta ]}(X/A)\to D^{[\beta +1]}(X/A)\]
along $\rho$. Here, we have let $\Delta [n-k]\to D(X/A)$ be the map that sends the identity $[n-k]\to [n-k]$ to $p^\beta (y)\mu$, where $\mu :[n-k]\to [n]$ is the section of $\rho$ given by $0\mapsto 0$. The map $\Delta [n-k]\to D(X/A)$ sends $\rho :[n]\to [n-k]$ to
\[(p^\beta (y)\mu )\rho =((w\rho )\mu )\rho =(w(\rho \mu ))\rho =w\rho =p^\beta (y).\]
Thus the solid diagram above commutes and we obtain canonical dashed maps $\chi ^{\beta +1}$ and $p^{\beta +1}$ as indicated. The observation that $p^\beta \circ \bar{y}$ factors through $N\rho$ is essentially a special case of \cite[Prop.~2.3.4.]{Fj20-DN}.

The map $f^{\beta ,\beta +1}$ is degreewise surjective as it is a cobase change of the degreewise surjective map $N\rho$. By the choice of $\rho$, the map $f^{\beta ,\beta +1}$ is a bijection in degree $0$ as the effect of taking the pushout along $\rho$ is trivial in degree $0$. Furthermore, the map $p^{\beta +1}$ is degreewise surjective as $p^\beta$ is. This shows that the second and third of the three conditions associated with the $(\beta +1)$-th stage are satisfied. However, the first remains to be verified.

Pushing out along $N\rho$ is not even useful unless $k>0$, for in that case the map $f^{\beta ,\beta +1}$ is an isomorphism. Moreover, we will, beginning with the next paragraph, argue that the vertices $y\varepsilon _{k+1}$, $\dots$, $y\varepsilon _n$ are pairwise distinct. As $y$ is non-embedded it will then follow that $k>0$. Notice that by the choice of $\rho$, the vertices of $z$ are pairwise distinct if the vertices $y\varepsilon _{k+1}$, $\dots$, $y\varepsilon _n$ are pairwise distinct. Thus it will follow that the simplex $z$ of $D^{[\beta +1]}(X/A)$ is embedded. In other words, to push out along $\rho$ is an optimal step in the desingularization process.

We prove that the vertices $y\varepsilon _{k+1} ,\dots ,y\varepsilon _n$ are pairwise distinct. First, note that the left hand square in the diagram
\begin{displaymath}
\xymatrix{
X \ar[rr]^(.4){f^{0,\beta } \circ \bar{f} } && D^{[\beta ]}(X/A) \ar[rr]^(.55){\chi ^\beta } && \Delta [1] \\
V \ar[u]^j \ar[rr]_{id_V} && V \ar[u]^{j^\beta } \ar[rr] && \Delta [0] \ar[u]_{N\varepsilon _1}
}
\end{displaymath}
is cartesian as both the outer and right hand squares are cartesian. As the map $f^{0,\beta } \circ \bar{f}$ is degreewise surjective, we can take the representing map $\Delta [n]\to X$ of some simplex $\tilde{y}$ of $X$ that $f^{0,\beta } \circ \bar{f}$ sends to $y$ and draw the diagram
\begin{displaymath}
\xymatrix{
\Delta [n] \ar[rr] \ar@/^3pc/[rrrr]^{\bar{y} } && X \ar[rr]^(.4){f^{0,\beta } \circ \bar{f} } && D^{[\beta ]}(X/A) \\
\Delta [n-k-1] \ar[u] \ar@{-->}[rr] \ar@/_2pc/[rrrr] && V \ar[u]^j \ar[rr]^{id_V} && V \ar[u]_{j^\beta }
}
\end{displaymath}

\bigskip
where we have pulled the representing map of $\tilde{y}$ back along $j$.

Note that the simplex $\tilde{y}$ is non-degenerate as $y$ is. Because $X$ is non-singular, it follows that the representing map of $\tilde{y}$ is degreewise injective. Therefore, its base change $\Delta [n-k-1]\to V$ along $j$ is degreewise injective. The outer square is cartesian as the left hand and right hand squares are cartesian. Hence, the composite of the two degreewise injective maps $j^\beta$ and $\Delta [n-k-1]\to V$ represents the $k$-th back face of $y$. Recall that $j^\beta$ is degreewise injective as it by assumption embeds $V$ as an abyss in $D^{[\beta ]}(X/A)$. This concludes our argument that the vertices $y\varepsilon _{k+1} ,\dots ,y\varepsilon _n$ are pairwise distinct. Recall that this implies that the simplex $z$ is embedded.

To form the diagram at the $(\beta +1)$-th stage of the sequence we define $i^{\beta +1}=f^{\beta ,\beta +1}\circ i^\beta$ and $j^{\beta +1}=f^{\beta ,\beta +1}\circ j^\beta$. This means that
\[\tilde{i} =p^\beta \circ i^\beta =(p^{\beta +1}\circ f^{\beta ,\beta +1})\circ i^\beta =p^{\beta +1}\circ (f^{\beta ,\beta +1}\circ i^\beta )=p^{\beta +1}\circ i^{\beta +1}\]
and that
\[\tilde{j} =p^\beta \circ j^\beta =(p^{\beta +1}\circ f^{\beta ,\beta +1})\circ j^\beta =p^{\beta +1}\circ (f^{\beta ,\beta +1}\circ j^\beta )=p^{\beta +1}\circ j^{\beta +1},\]
which shows that we get a diagram
\begin{displaymath}
\xymatrix{
&& \Delta [0] \ar@/_/[lld]_{\tilde{i} } \ar[d]_{i^{\beta +1}} \ar[rr] && \Delta [0] \ar[d]^{N\varepsilon _0} \\
D(X/A) && D^{[\beta +1]}(X/A) \ar[ll]_(.45){p^{\beta +1}} \ar[rr]^(.6){\chi ^{\beta +1}} && \Delta [1] \\
&& V \ar@/^/[llu]^{\tilde{j} } \ar[u]^{j^{\beta +1}} \ar[rr] && \Delta [0] \ar[u]_{N\varepsilon _1}
}
\end{displaymath}
together with a morphism from the $\beta$-th stage. It remains to argue that the two squares on the right are cartesian.

We can form pullbacks $C$ and $V'$ to obtain the diagram
\begin{displaymath}
\xymatrix{
\Delta [0] \ar[d]_{i^\beta } \ar@{-->}[rr] && C \ar[d] \ar[rr] && \Delta [0] \ar[d]^{N\varepsilon _0} \\
D^{[\beta ]}(X/A) \ar[rr]^(.45){f^{\beta ,\beta +1}} && D^{[\beta +1]}(X/A) \ar[rr]^(.55){\chi ^{\beta +1}} && \Delta [1] \\
V \ar[u]^{j^\beta } \ar@{-->}[rr] && V' \ar[u] \ar[rr] && \Delta [0] \ar[u]_{N\varepsilon _1}
}
\end{displaymath}
in which we by \cref{Lemma_Pullbacks_close_to_twooutofthree_property} get that the second and third quadrant are cartesian. The category $sSet$ has the property that a base change of a degreewise surjective map is again degreewise surjective. Consequently, the base changes $\Delta [0]\to C$ and $V\to V'$ of $f^{\beta ,\beta +1}$ must be degreewise surjective. Then $\Delta [0]\to C$ is trivially an isomorphism. In other words, the map $i^{\beta +1}$ is the base change of $N\varepsilon _0$ along $\chi ^{\beta +1}$.

It remains to argue that $V\to V'$ is degreewise injective. For this it suffices to argue that the composite
\[V\xrightarrow{j^\beta } D^{[\beta ]}(X/A)\xrightarrow{f^{\beta ,\beta +1} } D^{[\beta +1]}(X/A)\]
is degreewise injective. Take $m$-simplices $v$ and $w$ in $V$ and suppose
\[f^{\beta ,\beta +1}\circ j^\beta (v) =f^{\beta ,\beta +1}\circ j^\beta (w) .\]
We will prove that $v=w$. As $j^\beta$ is degreewise injective it is enough to prove that $j^\beta (v) =j^\beta (w)$. We can at least say that both of the simplices $j^\beta (v)$ and $j^\beta (w)$ are in the image of the representing map $\bar{y}$ or that $j^\beta (v)=j^\beta (w)$.

If the simplices $j^\beta (v)$ and $j^\beta (w)$ are in the image of $\bar{y}$, then there are operators
\[\alpha _v,\alpha _w:[m]\to [n]\]
such that $y\alpha _v=j^\beta (v)$ and $y\alpha _w=j^\beta (w)$. By our hypothesis we then know that
\begin{displaymath}
\begin{array}{rcl}
(\overline{z} \circ N\rho )\circ N\alpha _v & = & (f^{\beta ,\beta +1} \circ \overline{y} )\circ N\alpha _v \\
& = & f^{\beta ,\beta +1} \circ (\overline{y} \circ N\alpha _v) \\
& = & f^{\beta ,\beta +1} \circ (j^\beta \circ \bar{v} ) \\
& = & f^{\beta ,\beta +1} \circ (j^\beta \circ \bar{w} ) \\
& = & f^{\beta ,\beta +1} \circ (\overline{y} \circ N\alpha _w) \\
& = & (f^{\beta ,\beta +1} \circ \overline{y} )\circ N\alpha _w \\
& = & (\overline{z} \circ N\rho )\circ N\alpha _w.
\end{array}
\end{displaymath}
Given the fact that $z$ is embedded, the equation above implies
\[N\rho \circ N\alpha _v=N\rho \circ N\alpha _w\Rightarrow \rho \alpha _v=\rho \alpha _w.\]
Recall that, by definition, the degeneracy operator $\rho$ is injective on the subset $\{ k+1,\dots ,n\}$ of its source.

Because $y\alpha _v=j^\beta (v)$ is in the image of $j^\beta$, it follows that the image of $\alpha _v$ is contained in $\{ k+1,\dots ,n\}$. Recall the definition of $k$ from the diagram (\ref{eq:diagram_proof_of_prop_desingularizing_after_collapsing_elysium_special}). Similarly, because $y\alpha _w=j^\beta (w)$ is in the image of $j^\beta$, it follows that the image of $\alpha _w$ is contained in $\{ k+1,\dots ,n\}$. The fact that $\rho$ is injective on this subset combined with the equation $\rho \alpha _v=\rho \alpha _w$ yields $\alpha _v=\alpha _w$. This concludes the verification that $j^{\beta +1}$ is base change of $N\varepsilon _1$ along $\chi ^{\beta +1}$ and thus the construction of the $(\beta +1)$-th stage.

It remains to argue that the iterative desingularization process eventually halts. We will use the indices $x^\beta$, $\beta \geq 0$, defined above.

Let $\lambda$ be a cardinal that is strictly greater than the cardinality of $(X/A)^\sharp$. Define $S$ as the set consisting of those $x^\beta$ with $\beta \leq \lambda$. This is a subset of $(X/A)^\sharp$. Then we can consider the injective function $S\to \lambda +1$ defined by $x^\beta \mapsto \beta$. If $\alpha <\beta$, then $x^\alpha$ is defined if $x^\beta$ is. In other words, $\alpha$ is in the image of $S\to \lambda +1$ if $\beta$ is. By the choice of $\lambda$, there is no surjective extension
\begin{displaymath}
\xymatrix@=1em{
S \ar[dd] \ar[dr] \\
& \lambda +1 \\
(X/A)^\sharp \ar@{-->}[ur]_\nexists
}
\end{displaymath}
of $S\to \lambda +1$ to $(X/A)^\sharp$. In other words, $S\to \lambda +1$ cannot possibly be surjective. Hence, the element $\lambda$ is not in the image of the latter function. By the definition of $S$ it follows that $x^\lambda$ is not defined, so the set $S$ contains all simplices of $X/A$ with a designation $x^\beta$. This shows that $D^{[\lambda ]}(X/A)$ is non-singular, so the method we use in order to desingularize $X/A$ does indeed come to a halt.

As a result we get that $p^\lambda :D^{[\lambda ]}(X/A)\xrightarrow{\cong } D(X/A)$ is an isomorphism. Now, the simplicial set $D^{[\lambda ]}(X/A)$ belongs to a diagram that displays $V$ embedded as an abyss in $D^{[\lambda ]}(X/A)$. By design, the composite
\[V\xrightarrow{j^\lambda } D^{[\lambda ]}(X/A)\xrightarrow{p^\lambda } D(X/A)\]
is a factorization of the canonical map $\tilde{\jmath } :V\to D(X/A)$, so this finishes our proof of \cref{prop:desingularizing_after_collapsing_elysium}.
\end{proof}

%% file: sections/properties.tex
\section{Properties of Str\o m maps}
\label{sec:properties}

In this section, we will prove that the class of Str\o m maps is closed under cobase change (in $nsSet$), stated as \cref{prop:Strom-maps_closed_under_cobasechange}. Based on this result, we establish \cref{lem:Pushout_along_strom_homotopically_wellbehaved}, which says that to take a pushout along a Str\o m map is a homotopically well behaved operation. The latter will be the key to establishing the model structure on $nsSet$ and to the relationship with the model category of simplicial sets.

First, consider the following lemma.
\begin{lemma}
\label{lem_tool_for_proving_lem_Strom-maps_closed_under_cobasechange}
Suppose $k:A\to B$ the inclusion of an eden $A$ in a non-singular simplicial set $B$ and that $f:A\to C$ is some map in $nsSet$. Assume that there is an abyss $W$ in $B$ that contains $A$. Let $i$ denote the inclusion $A\to W$ and let $j$ denote the inclusion $W\to B$. Then the canonical map
\[B\sqcup _WD(W\sqcup _AC)\xrightarrow{\cong } D(B\sqcup _AC)\]
is an isomorphism.
\end{lemma}
\noindent The proof of \cref{lem_tool_for_proving_lem_Strom-maps_closed_under_cobasechange} is an adaptation of Thomason's argument on page 315 in his article \cite{Th80} whose purpose is analogous.
\begin{proof}[Proof of \cref{lem_tool_for_proving_lem_Strom-maps_closed_under_cobasechange}.]
Let $V$ denote the full simplicial subset of $B$ whose $0$-simplices are those that are not simplices of $A$. Then $V$ is an abyss in $B$. Consider the square
\begin{displaymath}
\xymatrix{
V\cap W \ar[d] \ar[r] & W \ar[d] \\
V \ar[r] & B
}
\end{displaymath}
in $sSet$. The simplicial set $V\cap W$ is an abyss in both $V$ and $W$. Due to these facts and the fact that $B=V\cup W$, it follows that the square is cocartesian. We put it next to the diagram (\ref{eq:diagram_proof_of_lem_Strom-maps_closed_under_cobasechange_big}). Then we get a canonical isomorphism
\[B\sqcup _WD(W\sqcup _AC)\cong V\sqcup _{V\cap W}D(W\sqcup _AC)\]
between pushouts in $sSet$.

We know from \cref{prop:desingularizing_after_collapsing_elysium} that the canonical map
\[V\cap W\to D(W/A)\]
is an abyss, hence
\[V\cap W\to D(W\sqcup _AC)\]
is degreewise injective.
Therefore, the simplicial set $V\sqcup _{V\cap W}D(W\sqcup _AC)$ is the pushout in $sSet$ of a diagram in which all objects are non-singular and where both legs are degreewise injective, which means that the pushout is itself non-singular. By the universal property of desingularization, it follows that the canonical map
\[B\sqcup _WD(W\sqcup _AC)\xrightarrow{\cong } D(B\sqcup _AC)\]
is an isomorphism.
\end{proof}
\noindent Next, we combine \cref{lem_tool_for_proving_lem_Strom-maps_closed_under_cobasechange} with \cref{prop:desingularizing_after_collapsing_elysium} to establish \cref{prop:Strom-maps_closed_under_cobasechange}.

In the proof of \cref{lem:Pushout_along_strom_homotopically_wellbehaved} below, we will refer to the full strength of \cref{prop:Strom-maps_closed_under_cobasechange} and not just that Str\o m maps are closed under taking cobase change. Hence the slightly awkward formulation of \cref{prop:Strom-maps_closed_under_cobasechange}.
\begin{proposition}\label{prop:Strom-maps_closed_under_cobasechange}
The class of Str\o m maps is closed under taking cobase change (in $nsSet$). Moreover, if $k:A\to B$ is a Str\o m map with factorization
\[A\xrightarrow{i} W\xrightarrow{j} B\]
and if the diagram
\begin{displaymath}
\xymatrix@C=1em{
  A \ar@/_1.5pc/[dd]_k \ar[r]^f \ar[d]^i & C \ar[d]_{\hat{\imath }} \ar@/^3.5pc/[dd]^{\hat{k} } \\
  W \ar[r] \ar[d]^j & D(W\sqcup _AC) \ar[d]_{\hat{\jmath }} \\
  B \ar[r] & D(B\sqcup _AC) \\
}
\end{displaymath}
in $nsSet$ displays $\hat{k}$ as the cobase change of $k$ along some map $f:A\to C$ and $\hat{\imath }$ as the cobase change of $i$ along $f$, then
\[A\xrightarrow{\hat{\imath }} W\xrightarrow{\hat{\jmath } } B\]
is a factorization of $\hat{k}$ as a Str\o m map.
\end{proposition}
\begin{proof}
Consider the commutative diagram
\begin{equation}
\label{eq:diagram_proof_of_lem_Strom-maps_closed_under_cobasechange_big}
\begin{gathered}
\xymatrix@C=1em{
  A \ar@/_1pc/[dd]_k \ar[r]^f \ar[d]^i & C \ar[d]^{\bar{\imath }} \ar[dr]^{\hat{\imath }} \ar[rrr] &&& \Delta [0] \ar[d] \\
  W \ar[r]^(.35)g \ar[d]^j & W\sqcup _A C \ar[d]^{\bar{\jmath }} \ar[r] & D(W\sqcup _AC) \ar[dr]^{\hat{\jmath }} \ar[rr] \ar[d] && D(W/A) \ar[d]\\
  B \ar[r]^(.35)h & B\sqcup _AC \ar@/_2pc/[rr]_{\eta _{B\sqcup _AC}} \ar[r] & B\sqcup _WD(W\sqcup _AC) \ar[r]_(.6)\cong & D(B\sqcup _AC) \ar[r] & D(B/A) \\
}
\end{gathered}
\end{equation}
\\
in $sSet$, where we have used the naturality of $W\sqcup _AC\to D(W\sqcup _AC)$. Because we simplify notation many places, for instance by removing redundant $U$'s, the terms natural and naturality may seem out of place. Nevertheless, it is the category-theoretical notion that is understood. Notice that the cobase change $\hat{k} =\hat{\jmath } \circ \hat{\imath }$ of $k$ in $nsSet$ is present in the diagram, diagonally.

\cref{def:strom} has four conditions that the map $\hat{k}$ must satisfy. We will start by confirming the third, which is that there is a retraction
\[\hat{r} :D(W\sqcup _AC)\to C\]
of $\hat{\imath }$. This is immediate from the existence of the retraction $r:W\to A$ of $i$ as we see in the diagram
\begin{displaymath}
\xymatrix{ %% Hvordan lage en pil fra W til A og til høyre for pilen fra A til W?
  A \ar[rr]^f \ar[d]_i && C \ar[d]_{\hat{\imath }} \ar@/^1pc/[ddr]^1 \\
  W \ar[dr]^r \ar[rr]^(.45){\eta \circ g} && D(W\sqcup _AC) \ar@{-->}[dr]_{\hat{r} } \\
  & A \ar[rr]^f && C
}
\end{displaymath}
in $nsSet$ where we make use of the universal property of $D(W\sqcup _AC)$ as a pushout. This concludes our verification of the third condition of \cref{def:strom}.

For the fourth condition of \cref{def:strom} one should be convinced that the functor
\[-\times \Delta [1]:nsSet\to nsSet\]
preserves pushouts, which it does according to \cite[Cor.~3.1.2.]{FR20}. Hence, the simplicial homotopy rel $A$ denoted $\epsilon$ that comes with the Str\o m map $k$ gives rise to a corresponding simplicial homotopy $\hat{\epsilon }$ via the diagram
\begin{equation}
\label{eq:diagram_proof_of_lem_Strom-maps_closed_under_cobasechange_existence_homotopy}
\begin{gathered}
\xymatrix@C-1pc@R-1pc{ %% Hvordan lage en pil fra W til A og til høyre for pilen fra A til W?
  A\times \Delta [1] \ar[rr]^{f\times 1} \ar[dd]_{i\times 1} && C\times \Delta [1] \ar[dd]_{\hat{i} \times 1} \ar[dr]^{pr_1} \\
  &&& C \ar[dd]_{\hat{i}} \\
  W\times \Delta [1] \ar[rr]^(.4){(\eta \circ g)\times 1} \ar[dr]^\epsilon && D(W\sqcup _AC)\times \Delta [1] \ar@{-->}[dr]^(.6){\hat{\epsilon } } \\
  & W \ar[rr]_{\eta \circ g} && D(W\sqcup _AC)
}
\end{gathered}
\end{equation}
in $nsSet$. We can expand the diagram by considering the diagram
\begin{displaymath}
\xymatrix{
W \ar[d]^{i_0} & A \ar[l]_i \ar[d]^{i_0} \ar[r]^f & C \ar[d]^{i_0} \\
W\times \Delta [1] & A\times \Delta [1] \ar[l]_{i\times 1} \ar[r]^{f\times 1} & C\times \Delta [1] \\
W \ar[u]_{i_1} & A \ar[l]^i \ar[u]_{i_1} \ar[r]_f & C \ar[u]_{i_1}
}
\end{displaymath}
in $nsSet$. It gives rise to a diagram
\begin{displaymath}
\xymatrix{
D(W\sqcup _AC) \ar[d]_{i_0} \ar[dr] \\
D(W\sqcup _AC)\times \Delta [1] \ar[r]^(.6){\hat{\epsilon } } & D(W\sqcup _AC) \\
D(W\sqcup _AC) \ar[u]^{i_1} \ar[ur]_{id}
}
\end{displaymath}
in which the composite $\hat{\epsilon } \circ i_1$ is the identity. Using the universal property of $D(W\sqcup _AC)$, one can check that the upper diagonal map
\[D(W\sqcup _AC)\to D(W\sqcup _AC)\]
is $\hat{i} \circ \hat{r}$. Thus $\hat{\epsilon }$ is a deformation of $D(W\sqcup _AC)$ to $C$. That the deformation is rel $C$ is immediate from the diagram that defines $\hat{\epsilon }$, namely (\ref{eq:diagram_proof_of_lem_Strom-maps_closed_under_cobasechange_existence_homotopy}). This concludes our verification of the fourth condition of \cref{def:strom}.

We are about to take care of the first and the second condition of \cref{def:strom}. To this end, note that \cref{lem_tool_for_proving_lem_Strom-maps_closed_under_cobasechange} below says that the canonical map
\[B\sqcup _WD(W\sqcup _AC)\xrightarrow{\cong } D(B\sqcup _AC)\]
is an isomorphism. This implies that the map $\hat{\jmath }$ is identified with a map that is a cobase change in $sSet$ of the abyss $j$. Thus $\hat{\jmath }$ is an abyss. In other words, the second condition of \cref{def:strom} holds.

In particular, the map $\hat{\jmath }$ is degreewise injective. Hence, the map $\hat{k}$ is degreewise injective, for it is the composite $\hat{\jmath } \circ \hat{\imath }$. Recall that the map $\hat{\imath }$ is degreewise injective as it is a section of $\hat{r}$.

Finally, we prove that the first condition of \cref{def:strom} holds. By \cref{lem:elysiums_abysses_preserved_cobase_change}, the cobase change $\bar{k} =\bar{\jmath } \circ \bar{\imath }$ in $sSet$ of $k$ is an eden. Furthermore, the characteristic map $\chi :B\sqcup _AC\to \Delta [1]$ of $C$ as an eden in $B\sqcup _AC$ gives rise to a unique map
\[\Psi :D(B\sqcup _AC)\to \Delta [1]\]
such that $\chi =\Psi \circ \eta _{B\sqcup _AC}$ via the universal property of desingularization. We will argue that $\Psi$ is the characteristic map of $C$ as an eden in $D(B\sqcup _AC)$, meaning that $\hat{k}$ is the base change of $N\varepsilon _0$ along $\Psi$.

Suppose we are given a simplicial set $X$ and maps $\beta :X\to B$ and $\gamma :X\to C$ such that
\begin{equation}
\label{eq:proof_first_condition_Strom-maps_closed_under_cobasechange}
\hat{k} \circ \gamma =\eta _{B\sqcup _AC}\circ \beta .
\end{equation}
Consider the solid arrow diagram
\begin{displaymath}
\xymatrix{
X \ar@/_1pc/[ddr]_\beta \ar@{-->}[dr]_\alpha \ar@/^1pc/[drr]^\gamma \\
& C \ar[d]^{\bar{k} } \ar[r]_{id_C} & C \ar[d]^(.45){\hat{k} } \ar[r]_h & \Delta [0] \ar[d]^{N\varepsilon _0} \\
& B\sqcup _AC \ar[r]_(.45){\eta _{B\sqcup _AC}} & D(B\sqcup _AC) \ar[r]^(.6)\Psi & \Delta [1]
}
\end{displaymath}
in $sSet$. Notice from the equations
\begin{displaymath}
\begin{array}{rcl}
N\varepsilon _0\circ h & = & N\varepsilon _0\circ h\circ id_C \\
& = & \chi \circ \bar{k} \\
& = & (\Psi \circ \eta _{B\sqcup _AC})\circ \bar{k} \\
& = & \Psi \circ (\eta _{B\sqcup _AC}\circ \bar{k} ) \\
& = & \Psi \circ (\hat{k} \circ id_C) \\
& = & \Psi \circ \hat{k}
\end{array}
\end{displaymath}
that the right hand square commutes.

We use that the outer square is cartesian to obtain a dashed map $\alpha :X\to C$ such that
\begin{displaymath}
\begin{array}{rcl}
\beta & = & \bar{k} \circ \alpha \\
h\circ \gamma & = & (h\circ id_C)\circ \alpha .
\end{array}
\end{displaymath}
The second equation is uninteresting, but the first combined with (\ref{eq:proof_first_condition_Strom-maps_closed_under_cobasechange}) yields
\[\hat{k} \circ \gamma =\eta _{B\sqcup _AC}\circ \beta =\eta _{B\sqcup _AC}\circ (\bar{k} \circ \alpha )=(\eta _{B\sqcup _AC}\circ \bar{k} )\circ \alpha =\hat{k} \circ \alpha .\]
Thus $\alpha =\gamma$ as $\hat{k}$ is degreewise injective. The degreewise injective maps are the monomorphisms of $sSet$. This shows that the left hand square is cartesian.

Because $\eta _{B\sqcup _AC}$ is degreewise surjective it follows by \cref{cor:sSet_Pullbacks_close_to_twooutofthree_property} that the right hand square is cartesian. In other words, the map $\hat{k}$ is the base change of $N\varepsilon _0$ along $\Psi$. This concludes our verification of the first condition of \cref{def:strom}.
\end{proof}
\noindent The proof of \cref{prop:Strom-maps_closed_under_cobasechange} finishes the technical bulk of this article.

We conclude the section by establishing the following crucial homotopical link between simplicial sets and non-singular simplicial sets. It is an adaptation of the analogous result for Dwyer maps \cite[Prop.~4.3]{Th80}.
\begin{lemma}\label{lem:Pushout_along_strom_homotopically_wellbehaved}
Let $k:A\rightarrow B$ be a Str\o m map and $f:A\rightarrow C$ some map in $nsSet$. If the square
\begin{displaymath}
\xymatrix{
A \ar[d]_k \ar[r]^f & C \ar[d] \\
B \ar[r] & D(UB\sqcup _{UA}UC)
}
\end{displaymath}
is cocartesian in $nsSet$, then the square
\begin{displaymath}
\xymatrix{
UA \ar[d]_{Uk} \ar[r]^{Uf} & UC \ar[d] \\
UB \ar[r] & UD(UB\sqcup _{UA}UC)
}
\end{displaymath}
is homotopy cocartesian in $sSet$.
\end{lemma}
\begin{proof}
We are pedantic in the formulation of the proposition in the hope that the notation will make it clear which pushout belongs in which category. What we will prove is that the canonical map
\[UB\sqcup _{UA}UC\to UD(UB\sqcup _{UA}UC)\]
from the pushout in $sSet$ of the diagram
\[UB\xleftarrow{Uk} UA\xrightarrow{Uf} UC\]
to the pushout in $nsSet$ of the underlying diagram is a weak equivalence in $sSet$. Now, we remove the redundant $U$'s from the notation and proceed.

Suppose $k=j\circ i$ a factorization of $k$ as a Str\o m map. Assume that $\hat{k} =\hat{\jmath } \circ \hat{\imath }$ is the cobase change in $nsSet$ of $k$ along $f$ and that $\hat{\imath }$ is the cobase change in $nsSet$ of $i$ along $f$. By \cref{prop:Strom-maps_closed_under_cobasechange}, it follows that the right hand vertical map in the diagram
\begin{displaymath}
\xymatrix{
	B \ar[d]_1 & A \ar[l]_k \ar[d]_i^\sim \ar[r]^f & C \ar[d]_{\hat{i} }^\sim \\
	B & W \ar[l]_j \ar[r] & D(W\sqcup _AC)
}
\end{displaymath}
in $sSet$ is a weak equivalence. The diagram yields a factorization of
\[\eta _{B\sqcup _AC}:B\sqcup _AC\to D(B\sqcup _AC)\]
as
\[B\sqcup _AC\xrightarrow{\sim } B\sqcup _WD(W\sqcup _AC)\xrightarrow{\cong } D(B\sqcup _AC).\]
Here, the first map is a weak equivalence by the glueing lemma\\ \cite[Prop.~13.3.9, p.~246]{Hi03}. Note that $k$ and $j$ are cofibrations in the standard model structure on $sSet$ as the cofibrations are the degreewise injective maps. The second map is an isomorphism by \cref{lem_tool_for_proving_lem_Strom-maps_closed_under_cobasechange}.
\end{proof}

%% file: sections/lifting.tex
\section{Lifting conditions}
\label{sec:lifting}

In this section, we finally verify the lifting conditions stated in \cref{thm:lifting_across_adjunction}, in the case when
\[(F,G)=(DSd^2,Ex^2U)\]
and when $sSet$ has the standard model structure. For this and the remaining part of this paper we need some more notation and terminology.

First, the following standard notation is convenient.
\begin{notation}\label{not:injectives_projectives_cofibrations_relations}
If $K$ is a class of maps in some category, then $K-inj$ denotes the class of maps $p$ such that $(i,p)$ is a lifting-extension pair for all members $i$ of $K$. Similarly, we let $K-proj$ denote the class of maps $i$ such that $(i,p)$ is a lifting extension pair for all members $p$ of $K$. Let
\[K-cof=(K-inj)-proj.\]
Expressed another way, the $K$-cofibrations are the maps that have the LLP with respect to the maps that have the RLP with respect to the members of $K$.
\end{notation}
\noindent Whenever one uses Hirschhorn's or Hovey's notion of cofibrantly generated model category, $K$-cof is the class of cofibrations if $K$ is a set of generating cofibrations. Similarly, $K$-cof is the class of trivial cofibrations if $K$ is a set of generating trivial cofibrations.

Suppose $X$ a $\lambda$-sequence for some $\lambda$. If $\mathscr{D}$ is a class of maps in $\mathscr{C}$ and if $X^{[\beta ]}\to X^{[\beta +1]}$ is a member of $\mathscr{D}$ whenever $\beta +1<\lambda$, then we say that $X$ is a \textbf{$\lambda$-sequence of maps in $\mathscr{D}$}. In such a case, consider a choice $f$ of a composition of $X$. We say that $X$ is a \textbf{presentation of $f$ (as a composition of maps in $\mathscr{D}$)} or that $X$ \textbf{presents $f$ (as a composition of maps in $\mathscr{D}$)}.
\begin{definition}\label{def:relative_cell_complex}
Let $K$ be a set of maps in a cocomplete category $\mathscr{C}$. A \textbf{relative $K$-cell complex} is a map that can be presented as a composition of maps in the class of cobase changes of maps taken from the set $K$. The class of relative $K$-cell complexes is denoted $K$-cell.
\end{definition}
\noindent The class of relative $K$-cell complexes, denoted $K$-cell, is a subcategory of $\mathscr{C}$, but it is in fact far more flexible than that, as we now explain.

Any given composition of cobase changes of coproducts of maps from $K$ is a relative $K$-cell complex \cite[Prop.~10.2.14]{Hi03}. Furthermore, any given composition of relative $K$-cell complexes is again a relative $K$-cell complex \cite[Prop.~10.2.15]{Hi03}.

The members of $K$-cof are called \textbf{$K$-cofibrations}. Note that
\[K-cell\subseteq K-cof\]
according to the general theory \cite[Prop.~10.5.10]{Hi03}. The relative $K$-cell complexes, typically, have more in common with the members of $K$ than the $K$-cofibrations have in common with memebers of $K$. This is because the flexibility of $K$-cell tends to make properties of members of $K$ carry over to relative $K$-cell complexes, whereas the same properties can fail to carry over from relative $K$-cell complexes to $K$-cofibrations. If, however, $K$ is a set of generating (resp. trivial) cofibrations for a model category, then the class $K$-cof of (resp. trivial) cofibrations equals the class of retracts of relative $K$-cell complexes \cite[Prop.~11.2.1, p.~211]{Hi03}. The set $K$ is generally thought of as prototypes of the (resp. trivial) cofibrations.

The following terminology will be convenient in the verification of the first condition of \cref{thm:lifting_across_adjunction}.
\begin{definition}\label{def:def_comp_of_strom_maps}
A composition in $nsSet$ of maps in the class of Str\o m maps is referred to as a \textbf{composition of Str\o m maps}.
\end{definition}
\noindent Note that if the members of a certain class have a common name, then we might use that name along the lines of \cref{def:def_comp_of_strom_maps}.

Recall \cref{not:prototypes_cofibr_trivial_cofibr_standard_pre_exist}. The symbol $J-inj$ refers to the class of fibrations in $sSet$ equipped with the standard model structure. Similarly, $I-inj$ is the class of trivial fibrations in $sSet$. Furthermore, $I-cof$ is the class of cofibrations and $J-cof$ is the class of trivial cofibrations in $sSet$. The examples above are immediate from Proposition 11.2.1 in Hirschhorn's book \cite[p.~211]{Hi03}.
\begin{lemma}\label{lem:relative_cell_complexes_degreewise_injective}
Each relative $DSd^2(I)$-cell complex or relative $DSd^2(J)$-cell complex is a composition of Str\o m maps. In particular, every member of each of these classes of relative cell complexes is degreewise injective when viewed as a map in $sSet$.
\end{lemma}
\begin{proof}
The members of $DSd^2(I)$ and $DSd^2(J)$ are Str\o m maps by \cref{cor:two-fold_subdivision_strom}. The class of Str\o m maps is closed under taking cobase change by \cref{prop:Strom-maps_closed_under_cobasechange}. Therefore, any relative $DSd^2(I)$-cell complex or relative $DSd^2(J)$-cell complex is a composition of Str\o m maps.

Let $j$ be a composition of Str\o m maps. Then $U(j)$ is a composition in $sSet$ of degreewise injective maps, as $U:nsSet\to sSet$ preserves filtered colimits. Hence $U(j)$ is itself degreewise injective.
\end{proof}
\noindent With \cref{lem:relative_cell_complexes_degreewise_injective} and the terminology we have so far, we are ready to verify the second condition stated in \cref{thm:lifting_across_adjunction}.

The proof of \cref{prop:second_condition_lifting_criterion} is built on a technique taken from Thomason \cite{Th80}, although more people deserve credit for the ideas that are involved, such as A. Str\o m who worked with characterizations of cofibrations in model structures on topological spaces, and also people developing the theory of neighborhood deformation retracts.
\begin{proposition}\label{prop:second_condition_lifting_criterion}
Let $f$ be a relative $DSd^2(J)$-cell complex. Then $U(f)$ is a weak equivalence.
\end{proposition}
\begin{proof}
Suppose
\begin{equation}
\label{eq:first_diagram_proof_prop_second_condition_lifting_criterion}
\begin{gathered}
\xymatrix@=0.9em{
A=A^{[0]} \ar@/_/[dr]_(.3)f \ar[r] & A^{[1]} \ar[d] \ar[r] & \dots \ar[r] & A^{[\beta ]}\ar@/^/[lld] \ar[r] & \dots \\
& B=colim_{\beta <\lambda }A^{[\beta ]}
}
\end{gathered}
\end{equation}
a presentation of $f$. By \cref{lem:relative_cell_complexes_degreewise_injective}, the map $f$ is a composition of Str\o m maps. The functor $U$ preserves filtered colimits (say by \cite[Lem.~5.1.2.]{Fj18}), so the $\lambda$-sequence $U\circ A$ is a presentation of $Uf$ as a composition of inclusions of Str\o m maps.

Suppose the diagram
\begin{displaymath}
\xymatrix{
U\Lambda \ar[d]^\sim \ar[r] & UA^{[\beta ]}\ar[d] \ar@/^1pc/[ddr] \\
U\Lambda ' \ar[r] \ar@/_1pc/[drr] & \Lambda '' \ar[dr]^\sim \\
&& UA^{[\beta +1]}
}
\end{displaymath}
in $sSet$ displays $A^{[\beta ]}\to A^{[\beta +1]}$ the way it arises as a cobase change in $nsSet$ of some element $\Lambda \to \Lambda '$ of the set $DSd^2(J)$. Here, the simplicial set $\Lambda ''$ denotes the pushout in $sSet$, $A^{[\beta +1]}$ denotes the pushout in $nsSet$ and the map $\Lambda ''\xrightarrow{\sim } UA^{[\beta +1]}$ is the canonical map, which is a weak equivalence due to \cref{lem:Pushout_along_strom_homotopically_wellbehaved}.

The cobase change $UA^{[\beta ]}\to \Lambda ''$ in $sSet$ is a trivial cofibration as $U\Lambda \to U\Lambda '$ is a trivial cofibration. Consequently, the inclusion $UA^{[\beta ]}\xrightarrow{\sim } UA^{[\beta +1]}$ of the cobase change in $nsSet$ of $\Lambda \to \Lambda '$ is a composite of two weak equivalences and therefore itself a weak equivalence. Moreover, the map $UA^{[\beta ]}\xrightarrow{\sim } UA^{[\beta +1]}$ is degreewise injective as it is the result of applying $U$ to a Str\o m map. Thus we see that it is a trivial cofibration in the model category $sSet$, or in other words that it belongs to $J$-cof. The class $J$-cof is closed under taking compositions \cite[Lem.~10.3.1]{Hi03}. Therefore $U(f)$ is in $J$-cof and is in particular a weak equivalence.
\end{proof}
\noindent \cref{prop:second_condition_lifting_criterion} essentially takes care of the second condition stated in \cref{thm:lifting_across_adjunction}, which leaves the first condition.

Before we verify the first lifting condition, we introduce a bit more terminology.
\begin{definition}\label{def:regular_cardinal}
A cardinal $\kappa$ is said to be \textbf{regular} if, whenever $A$ is a set whose cardinal is less than $\kappa$ and for every $a\in A$ there is a set $S_a$ whose cardinal is less than $\kappa$, then the cardinal of $\bigcup _{a\in A}S_a$ is less than $\kappa$.
\end{definition}
\noindent For example, the countable cardinal $\aleph _0$ is regular \cite[Ex.~10.1.12]{Hi03}. Infinite successor cardinals are regular \cite[Prop.~10.1.14]{Hi03}.
\begin{definition}\label{def:smallness}
Assume that $\mathscr{C}$ is a cocomplete category, $\mathscr{D}$ a subcategory, $A$ an object and $\kappa$ a cardinal. We say that $A$ is \textbf{$\kappa$-small relative to $\mathscr{D}$} if we, for any given regular cardinal $\lambda \geq \kappa$, have that the covariant hom functor $\mathscr{C} (A,-):\mathscr{C} \to Set$ preserves the colimit of any given $\lambda$-sequence
\[X^{[0]}\to \dots \to X^{[\beta ]}\to \dots\]
in $\mathscr{C}$ such that $X^{[\beta ]}\to X^{[\beta +1]}$ is a map of $\mathscr{D}$ whenever $\beta +1<\lambda$. We say that $A$ is \textbf{small relative to $\mathscr{D}$} if it is $\kappa$-small relative to $\mathscr{D}$ for some $\kappa$.
\end{definition}
\noindent We state the following example concerning the category $sSet$.
\begin{example}\label{ex:ssets_small}
If $X$ is a simplicial set and $\kappa$ is the first infinite cardinal that is greater than the cardinal of the set $X^\sharp$ of non-degenerate simplices, then $X$ is $\kappa$-small relative to the subcategory of degreewise injective maps.
\end{example}
\noindent A reference for the fact presented in \cref{ex:ssets_small} is Ex.~10.4.4 from \cite[pp.~194]{Hi03}.

The following remark may be in order.
\begin{remark}
No argument for Hirschhorn's smallness result \cite[Ex.~10.4.4]{Hi03} is presented in his book. A similar statement can be formulated by combining Lemmas $3.1.1$ and $3.1.2$ in Hovey's book \cite[pp.~74]{Ho99}, or rather be extracted from the (sketches of) proofs of those lemmas. However, note that there is a slight difference in how Hirschhorn and Hovey defines smallness.

For comparison of Hovey's and Hirschhorn's notions of smallness, see Def.~2.1.3 in Hovey's book \cite[p.~29]{Ho99} and Def.~10.4.1 in Hirschhorn's book \cite[p.~194]{Hi03}.

The smallness result as stated by Hirschhorn appears weaker than Hovey's. Hirschhorn only claims that simplicial sets are small relative to the subcategory of degreewise injective maps. Hovey sketches a proof of the stronger statement that simplicial sets are small (relative to the category $sSet$ itself). It seems likely that Hovey's sketch can be adapted to Hirschhorn's notion of smallness.
\end{remark}
\noindent As explained, we follow Hirschhorn's treatment of the subject of model categories, including his notion of smallness.

As a consequence of \cref{ex:ssets_small}, we get the following result in our setting.
\begin{lemma}\label{lem:nssets_small}
If $A$ is a non-singular simplicial set and $\kappa$ is the first infinite cardinal that is greater than the cardinal of the set $A^\sharp$ of non-degenerate simplices, then $A$ is $\kappa$-small relative to the subcategory of maps $f$ such that $U(f)$ is degreewise injective.
\end{lemma}
\begin{proof}
Suppose $\lambda \geq \kappa$ regular. Let $X:\lambda \to nsSet$ be a $\lambda$-sequence of maps whose inclusions are degreewise injective. Consider the universal cocone
\begin{equation}
\label{eq:diagram_cocone_on_X_proof_of_lem_nssets_small}
\begin{gathered}
\xymatrix{
X^{[0]} \ar@/_/[dr] \ar[r] & X^{[1]} \ar[d] \ar[r] & \dots \ar[r] & X^{[\beta ]} \ar@/^/[lld] \ar[r] & \dots \\
& colim_{\beta <\lambda }X^{[\beta ]}
}
\end{gathered}
\end{equation}
on $X$. The cocone
\begin{displaymath}
\xymatrix{
UX^{[0]} \ar@/_/[dr] \ar[r] & UX^{[1]} \ar[d] \ar[r] & \dots \ar[r] & UX^{[\beta ]} \ar@/^/[lld] \ar[r] & \dots \\
& U(colim_{\beta <\lambda }X^{[\beta ]} )
}
\end{displaymath}
on $U\circ X$ is universal as the inclusion $U:nsSet\to sSet$ preserves filtered colimits (say by \cite[Lem.~5.1.2.]{Fj18}). We get the diagram
\begin{equation}
\label{eq:diagram_proof_of_lem_nssets_small}
\begin{gathered}
\xymatrix@=0.8em{
sSet(UA,UX^{[0]}) \ar@/_1pc/[dddr] \ar@/_/[dr] \ar[r] & \dots \ar[r] & sSet(UA,UX^{[\beta ]}) \ar@/^/[ld] \ar@/^1pc/[lddd] \ar[r] & \dots \\
& colim _{\beta <\lambda }sSet(UA,UX^{[\beta ]}) \ar@{-->}[dd]^\cong \\
\\
& sSet(UA,U(colim_{\beta <\lambda }X^{[\beta ]}))
}
\end{gathered}
\end{equation}
in the category of sets, where the canonical function is a bijection because $UA$ is $\kappa$-small relative to the subcategory of degreewise injective maps.

We have the equalities
\[sSet(UA,UX^{[\beta ]})=nsSet(A,X^{[\beta ]}),\]
for each $\beta$ with $0\leq \beta <\lambda$, and
\[sSet(UA,U(colim_{\beta <\lambda }X^{[\beta ]}))=nsSet(A,colim_{\beta <\lambda }X^{[\beta ]}),\]
as $U$ is a full inclusion. The diagram (\ref{eq:diagram_proof_of_lem_nssets_small}) is with these replacements a diagram in the category of sets that arises from the diagram (\ref{eq:diagram_cocone_on_X_proof_of_lem_nssets_small}) in $nsSet$, so the non-singular simplicial set $A$ must be $\kappa$-small relative to the subcategory of maps whose inclusions are degreewise injective.
\end{proof}
\noindent \cref{lem:nssets_small} is more or less what we will use to verify the second condition stated in \cref{thm:lifting_across_adjunction} whose language is as follows.
\begin{definition}\label{def:permits_small_object_argument}
If $K$ is a set of maps in some cocomplete category, then $K$ \textbf{permits the small object argument} if the sources of the elements of $K$ are small relative to $K$-cell.
\end{definition}
\noindent The terminology presented in \cref{def:permits_small_object_argument} is part of Hirschhorn's notion of \emph{cofibrantly generated} \cite[Def.~11.1.2]{Hi03}, which is a property of model categories.

Note that Hirschhorn's notion may differ from Hovey's \cite[Def.~2.1.17]{Ho99} as the two authors' notions of \emph{smallness} differ slightly. Compare Hovey's definition \cite[Def.~2.1.3]{Ho99} with Hirschhorn's \cite[Def.~10.4.1]{Hi03}.

We say that a simplicial set is \textbf{finite} if it is generated by finitely many simplices. A simplicial set is finite if and only if it has finitely many non-degenerate simplices.
\begin{lemma}\label{lem:first_condition_lifting_criterion}
Each finite non-singular simplicial set is $\aleph _0$-small relative to the subcategory of maps $f$ such that $U(f)$ is degreewise injective.
\end{lemma}
\begin{proof}
Let $A$ be a finite non-singular simplicial set. Then $\aleph _0$ is the first infinite cardinal greater than the cardinality of the set $A^\sharp$ of non-degenerate simplices. Due to \cref{lem:nssets_small}, the simplicial set $A$ is thus $\aleph _0$-small relative to the subcategory of maps $f$ such that $U(f)$ is degreewise injective.
\end{proof}
\begin{lemma}\label{lem:doubly_subdivided_inclusions_of_boundaries_and_horns_permit_small_object}
Each of the sets $DSd^2(I)$ and $DSd^2(J)$ permits the small object argument.
\end{lemma}
\begin{proof}
Recall the natural map $b_X:Sd\, X\to BX$ from \cref{lem:properties_of_b_X}. For each $n\geq 0$, the simplicial set
\[BSd(\partial \Delta [n])\cong Sd^2(\partial \Delta [n])\cong DSd^2(\partial \Delta [n])\]
is the nerve of the poset $Sd(\partial \Delta [n])^\sharp$ of non-degenerate simplices of $Sd(\partial \Delta [n])$. This poset is finite, so its nerve has finitely many non-degenerate simplices. Similarly, for each expression $0\leq k\leq n>0$, the simplicial set
\[BSd(\Lambda ^k[n])\cong Sd^2(\Lambda ^k[n])\cong DSd^2(\Lambda ^k[n])\]
is the nerve of the poset $Sd(\Lambda ^k[n])^\sharp$ of non-degenerate simplices of $Sd(\Lambda ^k[n])$. This poset is finite, so its nerve has finitely many non-degenerate simplices.

By \cref{lem:first_condition_lifting_criterion}, the non-singular simplicial set $DSd^2(\partial \Delta [n])$ is $\aleph _0$-small relative to the subcategory of maps $f$ such that $U(f)$ is degreewise injective. For every relative $DSd^2(I)$-cell complex $f$, the map $U(f)$ is degreewise injective, by \cref{lem:relative_cell_complexes_degreewise_injective}. Similarly, the non-singular simplicial set $DSd^2(\Lambda ^k[n])$ is $\aleph _0$-small relative to $DSd^2(J)$-cell.
\end{proof}
\noindent Finally, \cref{lem:doubly_subdivided_inclusions_of_boundaries_and_horns_permit_small_object} confirms the first condition stated in the lifting theorem.

The work done so far yields the announced right-induced model structure on $nsSet$.
\begin{proposition}\label{prop:main_homotopy_theory}
Equip $sSet$ with the standard model structure. There is a cofibrantly generated model structure on $nsSet$ with $DSd^2(I)$ (resp. $DSd^2(J)$) serving as a set of generating (resp. trivial) cofibrations. When $nsSet$ is equipped with this model structure, the adjunction $(DSd^2,Ex^2U)$ is a Quillen pair.
\end{proposition}
\begin{proof}
We will apply \cref{thm:lifting_across_adjunction} to $(F,G)=(DSd^2,Ex^2U)$. First, note that $nsSet$ is bicomplete, by \cite[Cor.~2.2.3.]{Fj20-DN}. Now, consider the two conditions stated in the theorem.

The first condition holds by \cref{lem:doubly_subdivided_inclusions_of_boundaries_and_horns_permit_small_object}. As $Ex$ preserves and reflects weak equivalences, it follows from \cref{prop:second_condition_lifting_criterion} that the second condition also holds.
\end{proof}

%% file: sections/cofib.tex
\section{On cofibrations}
\label{sec:cofib}

The cofibrations in the cofibrantly generated model category $nsSet$ form the class $DSd^2(I)$-cof \cite[Prop.~11.2.1~(1)]{Hi03}. In this section, we will briefly discuss the $DSd^2(I)$-cofibrations and establish the important axiom of propriety, which in this case amounts to arguing that weak equivalences are preserved under cobase change along $DSd^2(I)$-cofibrations.

Notice that there is no change in the initial and terminal objects, compared with $sSet$.
\begin{lemma}
The empty simplicial set $\emptyset$ is the only initial object in the category $nsSet$. Similarly, the standard $0$-simplex $\Delta [0]$ a terminal object in $nsSet$.
\end{lemma}
\begin{proof}
The empty simplicial set $\emptyset$ is the colimit of the empty diagram in $sSet$. It is a non-singular simplicial set, so it is also the colimit of the underlying diagram in $nsSet$. Thus $\emptyset$ is initial in $nsSet$.

Similarly, the standard $0$-simplex $\Delta [0]$ is a limit of the empty diagram in $sSet$. Then $\Delta [0]$ is also the limit of the underlying diagram in $nsSet$ as this reflective subcategory inherits limits from $sSet$. Thus we can take $\Delta [0]$ to be a terminal object of $nsSet$.
\end{proof}
\noindent Furthermore, the following property of cofibrations is worth pointing out at this stage, although it is immediate from \cref{lem:relative_cell_complexes_degreewise_injective}.
\begin{lemma}\label{lem:cofib_degreewise_injective}
Any cofibration of $nsSet$ is a retract of a composition of Str\o m maps.
\end{lemma}
\noindent In particular, any cofibration is degreewise injective.
\begin{proof}[Proof of \cref{lem:cofib_degreewise_injective}.]
The cofibrations are precisely the retracts of the relative $DSd^2(I)$-cell complexes \\ \cite[Prop.~11.2.1.~(1),~p.~211]{Hi03}. From \cref{lem:relative_cell_complexes_degreewise_injective} we know that the relative $DSd^2(I)$-cell complexes are compositions of Str\o m maps, which are degreewise injective.
\end{proof}
\noindent Regrettably, \cref{lem:cofib_degreewise_injective} does not provide a characterization of the cofibrations of $nsSet$.

The following result concerns the classes $DSd^2(I)$-cell and $DSd^2(J)$-cell and is a strengthening of \cref{lem:Pushout_along_strom_homotopically_wellbehaved}.
\begin{lemma}\label{lem:pushout_along_transfinite_comp_of_strom}
Let $i:A\to B$ be a composition of Str\o m maps. Suppose $f:A\to C$ a map in $nsSet$. Then the canonical map
\[B\sqcup _AC\to D(B\sqcup _AC)\]
is a weak equivalence.
\end{lemma}
\noindent In previous sections, there were only one notion of weak equivalence, namely the weak equivalences in $sSet$. However, now that $nsSet$ is established as a model category there are really two notions of weak equivalence --- one in each model category.

To avoid confusion, one might want to write the canonical map of \cref{lem:pushout_along_transfinite_comp_of_strom} as
\[UB\sqcup _{UA}UC\to UD(UB\sqcup _{UA}UC).\]
On the other hand, because a map in $nsSet$ is a weak equivalence if and only if the result of applying $U$ to it is a weak equivalence, it is not necessary to be so pedantic. We simply remind the reader that we have a convention that the notation $B\sqcup _AC$ always refers to a pushout in $sSet$, and not in $nsSet$. This is because the symbol $D(B\sqcup _AC)$ is readily available to denote the pushout in $nsSet$ of the underlying diagram.
\begin{proof}[Proof of \cref{lem:pushout_along_transfinite_comp_of_strom}.]
Suppose $i$ has the presentation
\begin{displaymath}
\xymatrix{
A=A^{[0]} \ar@/_/[dr]_(.4)i \ar[r] & A^{[1]} \ar[d] \ar[r] & \dots \ar[r] & A^{[\beta ]}\ar@/^/[lld] \ar[r] & \dots \\
& B=colim_{\beta <\lambda }A^{[\beta ]}
}
\end{displaymath}
which by definition includes the assumption that each map $A^{[\beta ]}\to A^{[\beta +1]}$, $\beta +1<\lambda$, is a Str\o m map. 

Again, because the inclusion $U:nsSet\to sSet$ preserves filtered colimits, the $\lambda$-sequence $U\circ A$ is a presentation of $U(i)$ as a composition of inclusions of Str\o m maps.

Next, consider the diagram
\begin{displaymath}
\xymatrix{
A \ar[d]_i \ar[r]^f & C \ar[d] \ar@/^1pc/[ddr] \\
B \ar[r] \ar@/_1pc/[drr] & B\sqcup _AC \ar@{-->}[dr] \\
&& D(B\sqcup _AC)
}
\end{displaymath}
in $sSet$ from which the canonical map arises. Notice that it is the colimit of the $\lambda$-sequence of diagrams
\begin{displaymath}
\xymatrix{
A^{[0]} \ar[d] \ar[r]^f & C \ar[d] \ar@/^1pc/[ddr] \\
A^{[\beta ]}\ar[r] \ar@/_1pc/[drr] & A^{[\beta ]}\sqcup _{A^{[0]}}C \ar[dr] \\
&& D(A^{[\beta ]}\sqcup _{A^{[0]}}C)
}
\end{displaymath}
in $sSet$.

For the purposes of an argument by induction, consider the diagram
\begin{equation}
\label{eq:diagram_proof_of_lem_pushout_along_transfinite_comp_of_strom}
\begin{gathered}
\xymatrix@C=0.9em{
A^{[0]}\sqcup _{A^{[0]}}C \ar[d]^\sim \ar[r] & A^{[1]}\sqcup _{A^{[0]}}C \ar[d]^\sim \ar[r] & A^{[2]}\sqcup _{A^{[0]}}C \ar[d] \ar[r] & \cdots \\
D(A^{[0]}\sqcup _{A^{[0]}}C) \ar[r] & D(A^{[1]}\sqcup _{A^{[0]}}C) \ar[r] & D(A^{[2]}\sqcup _{A^{[0]}}C) \ar[r] & \cdots
}
\end{gathered}
\end{equation}
in $sSet$, which gives rise to
\[B\sqcup _AC\to D(B\sqcup _AC),\]
as we have established. Notice that the horizontal maps in the upper part of the diagram are degreewise injective. We now explain that the horizontal maps in the lower part are also degreewise injective.

Each map $A^{[\beta ]}\to A^{[\beta +1]}$,
\[0\leq \beta ,\; \beta +1<\lambda,\]
is a Str\o m map. Because the square
\begin{displaymath}
\xymatrix{
A^{[\beta ]} \ar[d] \ar[r] & D(A^{[\beta ]}\sqcup _{A^{[0]}}C) \ar[d] \\
A^{[\beta +1]} \ar[r] & D(A^{[\beta +1]}\sqcup _{A^{[0]}}C)
}
\end{displaymath}
in $nsSet$ is cocartesian, each map
\[D(A^{[\beta ]}\sqcup _{A^{[0]}}C)\to D(A^{[\beta +1]}\sqcup _{A^{[0]}}C)\]
is also a Str\o m map by \cref{prop:Strom-maps_closed_under_cobasechange} and thus degreewise injective.

Assume that an ordinal $\gamma \leq \lambda$ is such that
\[A^{[\beta ]}\sqcup _{A^{[0]}}C\xrightarrow{\sim } D(A^{[\beta ]}\sqcup _{A^{[0]}}C)\]
for any $\beta <\gamma$.

In the case when $\gamma$ is a limit ordinal, then the map
\[A^{[\gamma ]}\sqcup _{A^{[0]}}C\to D(A^{[\gamma ]}\sqcup _{A^{[0]}}C)\]
arises as a map of colimits, from a truncated version of (\ref{eq:diagram_proof_of_lem_pushout_along_transfinite_comp_of_strom}). In that truncated version, all the vertical maps are weak equivalences.

Next, we intend to use Kan's fibrant replacement functor $Ex^\infty$ on the truncated version of (\ref{eq:diagram_proof_of_lem_pushout_along_transfinite_comp_of_strom}). See \cite[pp.~215--217]{FP90} or \cite[p.~182--188]{GJ09}. The construction $Ex^\infty$ is the result of iterating the right adjoint $Ex:sSet\to sSet$ of the Kan subdivision. The functor $Ex$ can be defined thus
\[Ex(X)_n=sSet(Sd(\Delta [n]),X).\]
Kan's fibrant replacement preserves degreewise injective maps, filtered colimits and comes with a natural (degreewise injective) weak equivalence $e^\infty _X:X\xrightarrow{\sim } Ex^\infty\, X$, implying that the functor also preserves weak equivalences.

% The following discussion says that $Ex^\infty$ preserves filtered colimits.
% https://mathoverflow.net/questions/235526/when-do-colimits-agree-with-homotopy-colimits

% The following discussion says that homotopy groups of simplicial sets preserve filtered colimits. But is this true, in general? I think John mentioned that it is true that potentially long sequential colimits are preserved when the maps in the sequence are degreewise injective.
% https://mathoverflow.net/questions/56166/do-homotopy-groups-always-commute-with-filtered-colimits

Applying $Ex^\infty$ to the trunctated version of (\ref{eq:diagram_proof_of_lem_pushout_along_transfinite_comp_of_strom}) yields a diagram of fibrant simplicial sets (Kan sets) where the horizontal maps are degreewise injective and where the vertical maps are weak equivalences. The simplicial homotopy groups respects the colimit of a sequence whenever the maps of the sequence are degreewise injective. It follows that
\[A^{[\gamma ]}\sqcup _{A^{[0]}}C\xrightarrow{\sim } D(A^{[\gamma ]}\sqcup _{A^{[0]}}C)\]
is a weak equivalence.

In the case when $\gamma =\beta +1$ is a successor ordinal, we consider the diagram
\begin{displaymath}
\xymatrix@C=0.8em{
A^{[0]} \ar[d] \ar[r]^{f} & C \ar[d] \\
A^{[\beta ]}\ar[d] \ar[r] & A^{[\beta ]}\sqcup _{A^{[0]}}C \ar[d] \ar[r]^\sim & D(A^{[\beta ]}\sqcup _{A^{[0]}}C) \ar[d] \ar@/^1pc/[ddr] \\
A^{[\beta +1]} \ar[r] & A^{[\beta +1]}\sqcup _{A^{[0]}}C \ar@/_1pc/[drr] \ar[r] & A^{[\beta +1]}\sqcup _{A^{[\beta ]}}D(A^{[\beta ]}\sqcup _{A^{[0]}}C) \ar@{-->}[dr]^\sim \\
&&& D(A^{[\beta +1]}\sqcup _{A^{[0]}}C)
}
\end{displaymath}
in $sSet$. Here, 
\[A^{[\beta ]}\sqcup _{A^{[0]}}C\xrightarrow{\sim } D(A^{[\beta ]}\sqcup _{A^{[0]}}C)\]
is a weak equivalence by the induction hypothesis. The dashed map is a weak equivalence by \cref{lem:Pushout_along_strom_homotopically_wellbehaved}.

Because the map
\[A^{[\beta ]}\sqcup _{A^{[0]}}C\to A^{[\beta +1]}\sqcup _{A^{[0]}}C\]
is degreewise injective, the map
\[A^{[\beta +1]}\sqcup _{A^{[0]}}C\xrightarrow{\sim } A^{[\beta +1]}\sqcup _{A^{[\beta ]}}D(A^{[\beta ]}\sqcup _{A^{[0]}}C)\]
is a weak equivalence as $sSet$ is left proper. Therefore, the composite
\[A^{[\beta +1]}\sqcup _{A^{[0]}}C\to D(A^{[\beta +1]}\sqcup _{A^{[0]}}C)\]
is a weak equivalence.

Thus far we know that the vertical maps of (\ref{eq:diagram_proof_of_lem_pushout_along_transfinite_comp_of_strom}) are all weak equivalences. If we use Kan's fibrant replacement $Ex^\infty$ again, then we get that
\[B\sqcup _AC\cong colim_{\beta <\lambda }A^{[\beta ]}\sqcup _{A^{[0]}}C\xrightarrow{\sim } colim_{\beta <\lambda }D(A^{[\beta ]}\sqcup _{A^{[0]}}C)\cong D(B\sqcup _AC)\]
is a weak equivalence.
\end{proof}
\noindent Note that the lemma we have just proven has implications for both relative $DSd^2(I)$-cell complexes and relative $DSd^2(J)$-cell complexes as these are all compositions of Str\o m maps.

A result related to \cref{lem:pushout_along_transfinite_comp_of_strom} is the following, which implies that $nsSet$ is left proper.
\begin{lemma}\label{lem:pushout_along_cofibration}
Let $i:A\to B$ be a cofbration in $nsSet$. Suppose $f:A\to C$ a map in $nsSet$. Then the canonical map
\[\eta _{B\sqcup _AC}:B\sqcup _AC\to D(B\sqcup _AC)\]
is a weak equivalence.
\end{lemma}
\begin{proof}
The model category $nsSet$ is cofibrantly generated by \cref{prop:main_homotopy_theory} and thus we can factor $i=qj$ as a relative $DSd^2(I)$-cell complex $j:A\to X$ followed by a trivial fibration $q:X\to B$. Thus $(i,q)$ is a lifting-extension pair, so we can lift in the square
\begin{displaymath}
\xymatrix{
A \ar[d]_i \ar[r]^j & X \ar[d]^\sim _q \\
B \ar[r]^1 \ar@{-->}[ur]^s & B
}
\end{displaymath}
to write $i$ as a retract of $j$. This is what is known as the retract argument \cite[Prop.~7.2.2, p.~110]{Hi03}.

Next, we use the construction above to draw the diagram
\begin{displaymath}
\xymatrix{
& A \ar@{-}@/_1.9pc/[ldd]_j\hole \ar[d]_i \ar[r]^f & C \ar[d] \\
& B \ar@/_3.4pc/[ddd]_1 \ar[dd]_s \ar[r] & B\sqcup _AC \ar[ddd]_1 \ar[r] & D(B\sqcup _AC) \ar[ddd]_1 \\
\ar@/_/[dr] \\
& X \ar[d]_q^\sim \\
& B \ar[r] & B\sqcup _AC \ar[r] & D(B\sqcup _AC)
}
\end{displaymath}
in $sSet$. We will expand this diagram to display $\eta _{B\sqcup _AC}$
as a retract of the weak equivalence $\eta _{X\sqcup _AC}$.

Form the pushout $X\sqcup _AC$ in $sSet$ and then use the naturality of $\eta _{B\sqcup _AC}$ to expand the diagram above to the diagram
\begin{displaymath}
\xymatrix{
& A \ar@{-}@/_1.9pc/[ldd]_j\hole \ar[d]_i \ar[rr]^f && C \ar[d] \ar@{-}@/^1.5pc/[dr]\hole \\
& B \ar@/_3.4pc/[ddd]_1 \ar[dd]_s \ar[rr]^(.4){\bar{f} } && B\sqcup _AC\ar@{-->}[dd]_{\bar{s} } \ar@/_1.5pc/@{-}[ldd]_1\hole \ar[rr] & \ar@/^1pc/[ldd] & D(B\sqcup _AC) \ar[dd] \ar@/^4pc/[ddd]^1 \\
\ar@/_/[dr] \\
& X \ar[d]_q^\sim \ar[rr]_(.3)g & \ar@/_1pc/[dr] & X\sqcup _AC \ar[rr]^\sim _{\eta _{X\sqcup _AC}} && D(X\sqcup _AC) \\
& B \ar[rr]_(.4){\bar{f} } && B\sqcup _AC \ar[rr]_{\eta _{B\sqcup _AC}} && D(B\sqcup _AC)
}
\end{displaymath}
in which $\eta _{X\sqcup _AC}$ is a weak equivalence by \cref{lem:pushout_along_transfinite_comp_of_strom} as $j$ is a composition of Str\o m maps.

From this point, we can use that
\[X\sqcup _AC\cong X\sqcup _B(B\sqcup _AC)\]
to obtain a canonical map $\bar{q} :X\sqcup _AC\to B\sqcup _AC$ between pushouts. By its origin, it has the property that $1=\bar{q} \circ \bar{s}$ and $\bar{f} \circ q=\bar{q} \circ g$.

Finally, the naturality of $\eta _{X\sqcup _AC}$ and the functorality of desingularization finishes our argument that $\eta _{B\sqcup _AC}$ is a retract of the weak equivalence $\eta _{X\sqcup _AC}$. Then by the retract axiom for model categories, it follows that the former is a weak equivalence as the latter is.
\end{proof}
\noindent \cref{lem:pushout_along_cofibration} lets us deduce that $nsSet$ is proper.
\begin{proposition}\label{prop:axiom_of_propriety}
The model category $nsSet$ is proper.
\end{proposition}
\begin{proof}[Proof of \cref{prop:axiom_of_propriety}]
The model category $nsSet$ is automatically right proper as $sSet$ with the standard model structure is proper \cite[Thm.~13.1.13, p.~242]{Hi03}. We prove that $nsSet$ is left proper and thus proper.

Let $i:A\to B$ be a cofibration in $nsSet$. Suppose $f:A\to C$ a weak equivalence in $nsSet$. We will prove that the cobase change of $f$ along $i$ is a weak equivalence. Consider the diagram
\begin{displaymath}
\xymatrix{
A \ar[d]_i \ar[r]^f_\sim & C \ar[d]_{\bar{\imath } } \ar@/^1pc/[ddr]^j \\
B \ar[r]^(.35){\bar{f} } \ar@/_1pc/[drr]_g & B\sqcup _AC \ar[dr]^\sim \\
&& D(B\sqcup _AC)
}
\end{displaymath}
in $sSet$. The map
\[\eta _{B\sqcup _AC}:B\sqcup _AC\xrightarrow{\sim } D(B\sqcup _AC)\]
is a weak equivalence in $sSet$ as $i$ is a cofibration in $nsSet$. This is by \cref{lem:pushout_along_cofibration}.

The map $i$ is degreewise injective by \cref{lem:cofib_degreewise_injective} and hence a cofibration in $sSet$. Therefore, by propriety of $sSet$ it follows that $\bar{f}$ is a weak equivalence in $sSet$. Thus the composite $g$ is a weak equivalence in $sSet$. It is the cobase change in $nsSet$ of $f$ along $i$. Thus $nsSet$ is left proper, as was announced.
\end{proof}
\noindent Note that left propriety implies that we have a glueing lemma in the model category $nsSet$ \cite[Prop.~13.3.9, p.~246]{Hi03}.

We conclude this section by making a remark concerning the status of the work on characterizing the cofibrations and cofibrant objects in $nsSet$.
\begin{remark}
It does not seem likely that every composition of Str\o m maps is a cofibration. However, the converse may be true. According to the general theory, the $DSd^2(I)$-cofibrations are precisely the retracts of the relative $DSd^2(I)$-cell complexes \cite[Cor.~10.5.23, p.~200]{Hi03}.

The author has conjectured that every cofibrant non-singular simplicial set that is the nerve of a small category is even the nerve of a poset. This is analogous to Thomason's result that every cofibrant category is a poset \cite[Prop.~5.7]{Th80}. For a justification of this conjecture and for empirical evidence, one can have a look at \cite[Ch.~8]{Fj18}.

On the other hand, May, Stephan and Zakharevich \cite[p.~13]{MSZ17} has found a six-element poset in the model structure on $PoSet$ due to Raptis \cite{Ra10} that is not cofibrant. Let $P$ denote this poset. Because the right adjoint of the functor $q:PoSet\to nsSet$ is fully faithful, the counit $qNP\xrightarrow{\cong } P$ is an isomorphism. As $q$ is a left Quillen functor, the poset $qNP$ is cofibrant if $NP$ is, so $NP$ cannot be cofibrant in $nsSet$.

Bruckner and Pegel \cite{BP16} have found several classes of posets that are cofibrant in the model structure on $PoSet$ due to Raptis \cite{Ra10}. Hence, to claim that the nerve of any element taken from any of Bruckner's and Pegel's classes are cofibrant in $nsSet$ does not contradict the current knowledge of Raptis' model category.
\end{remark}

%% file: sections/inverse.tex
\section{A homotopy inverse of the inclusion}
\label{sec:inverse}

In this section, we prove that the Quillen pair $(DSd^2,Ex^2U)$ is indeed a Quillen equivalence. This is stated as \cref{prop:homotopy_inverse} below. In other words, towards the end of this section, we have sufficient knowledge to establish \cref{thm:main_homotopy_theory}, which is our main result.

Intuitively, the first step towards establishing the Quillen equivalence is the following result.
\begin{proposition}\label{prop:desing_double_subdvision_homotopy_inverse}
Let $X$ be a simplicial set. The unit $Sd^2\, X\to UDSd^2\, X$ of the adjunction
\begin{displaymath}
 \xymatrix{
 sSet \ar@<+.7ex>[r]^D & nsSet \ar@<+.7ex>[l]^U
 }
\end{displaymath}
is a weak equivalence.
\end{proposition}
\begin{proof}
Consider the skeleton filtration
\[X^0\to X^1\to \cdots \to X^n\to \cdots \]
of $X$, given by successively attaching the non-degenerate $k$-simplices to the $(k-1)$-skeleton, $k>0$. Note that $Sd^2\, X^n$ can be built from $Sd^2\, X^{n-1}$ as the Kan subdivision preserves colimits and 
\newpage{}
\noindent degreewise injective maps \cite[Prop.~4.6.3~(i),~p.~200]{FP90}.

By naturality, the unit $Sd^2\, X\to UDSd^2\, X$ arises as a map between sequential colimits from the diagram
\begin{displaymath}
\xymatrix{
Sd^2\, X^0 \ar[d]^\cong \ar[r] & Sd^2\, X^1 \ar[d] \ar[r] & \cdots \ar[r] & Sd^2\, X^n \ar[d] \ar[r] & \cdots \\
UDSd^2\, X^0 \ar[r] & UDSd^2\, X^1 \ar[r] & \cdots \ar[r] & UDSd^2\, X^n \ar[r] & \cdots
}
\end{displaymath}
in $sSet$. This is because $D$ is a left adjoint and because $U:nsSet\to sSet$ preserves filtered colimits (say by \cite[Lem.~5.1.2.]{Fj18}).

If $Sd^2\, X^n\to UDSd^2\, X^n$ is a weak equivalence for each $n\geq 0$, then $Sd^2\, X\to UDSd^2\, X$ is a weak equivalence. Now, the map
\[Sd^2\, X^0\xrightarrow{\cong} UDSd^2\, X^0\]
is an isomorphism for every $X$, because every $0$-dimensional simplicial set is non-singular.

Suppose $n>0$ is such that $Sd^2\, X^{n-1}\to UDSd^2\, X^{n-1}$ is a weak equivalence.  Hence, the diagram
\begin{equation}
\label{eq:diagram_proof_of_prop_desing_double_subdvision_homotopy_inverse}
\begin{gathered}
\xymatrix@=1em{
Sd^2(\bigsqcup _{x\in X^\sharp _n}\Delta [n]) \ar[d]^\cong & Sd^2(\bigsqcup _{x\in X^\sharp _n}\partial \Delta [n]) \ar[l] \ar[d]^\cong \ar[r] & Sd^2\, X^{n-1} \ar[d]^\sim \\
UDSd^2(\bigsqcup _{x\in X^\sharp _n}\Delta [n]) & UDSd^2(\bigsqcup _{x\in X^\sharp _n}\partial \Delta [n]) \ar[l] \ar[r] & UDSd^2\, X^{n-1}
}
\end{gathered}
\end{equation}
in $sSet$ yields a factorization
\[Sd^2\, X^n\xrightarrow{\sim } Z\to UDSd^2\, X^n\]
of the unit $Sd^2\, X^n\to UDSd^2\, X^n$ as a map between the pushouts $Sd^2\, X^n$ and $Z$ in $sSet$ followed by a canonical map $Z\to UDSd^2\, X^n$.

By the glueing lemma, the map $Sd^2\, X^n\xrightarrow{\sim } Z$ is a weak equivalence as the two left hand horizontal maps of (\ref{eq:diagram_proof_of_prop_desing_double_subdvision_homotopy_inverse}) are degreewise injective.

The map
\[Sd^2(\bigsqcup _{x\in X^\sharp _n}\partial \Delta [n])\to Sd^2(\bigsqcup _{x\in X^\sharp _n}\Delta [n])\]
is a Str\o m map by \cref{cor:two-fold_subdivision_strom}. By \cref{lem:Pushout_along_strom_homotopically_wellbehaved} it therefore follows that $Z\xrightarrow{\sim } UDSd^2\, X^n$ is a weak equivalence.
\end{proof}
\noindent Thus we obtain the fact that the homotopy type is preserved when we apply desingularization to the double Kan subdivision of some simplicial set.

Our second step is to move from considering the adjunction $(D,U)$ to considering the adjunction $(DSd^2,Ex^2U)$.
\begin{lemma}\label{lem:unit_weak_eq}
The unit $\eta _X:X\to Ex^2UDSd^2X$ is in general a weak equivalence.
\end{lemma}
\noindent \cref{lem:unit_weak_eq} will follow from the bulk of the proof of \cref{prop:second_condition_lifting_criterion}. In the language of Fritsch and Latch \cite{FL81}, the construction $DSd^2$ is a homotopy inverse for the inclusion $U:nsSet\to sSet$.
\begin{proof}[Proof of \cref{lem:unit_weak_eq}.]
The unit of $(DSd^2,Ex^2U)$ is that of the composite adjunction
\begin{displaymath}
 \xymatrix{
 sSet \ar@<+.7ex>[r]^{Sd^2} & sSet \ar@<+.7ex>[l]^{Ex^2} \ar@<+.7ex>[r]^D & nsSet \ar@<+.7ex>[l]^U
 }
\end{displaymath}
and is therefore itself the composite
\begin{equation}
\label{eq:diagram_proof_of_lem_unit_weak_eq}
\begin{gathered}
\xymatrix{
X \ar[r]^(.3)\sim & Ex^2(Sd^2\, X) \ar[r] & Ex^2(UD(Sd^2\, X))
}
\end{gathered}
\end{equation}
where the first map is known to be a weak equivalence. To see that the latter statement is true, it is enough to realize that the unit $X\to ExSd\, X$ of $(Sd,Ex)$ is a weak equivalence.

Adjoint \cite[p.~213]{FP90} to the last vertex map $d_X:Sd\, X\xrightarrow{\sim } X$ is a natural weak equivalence $e_X:X\xrightarrow{\sim } Ex\, X$ \cite[Lem.~4.6.20]{FP90}. The unit of $(Sd,Ex)$ is adjoint to the identity $Sd\, X\to Sd\, X$. Moreover, the unit of $(Sd,Ex)$ fits into the commutative triangle
\begin{displaymath}
\xymatrix{
& ExSd\, X \ar[dr]^{Ex(d_X)}_\sim \\
X \ar[ur] \ar[rr]_{e_X}^\sim && Ex\, X
}
\end{displaymath}
as we see from the commutative square
\begin{displaymath}
\xymatrix{
sSet(Sd\, X,Sd\, X) \ar[d]_{sSet(id,d_X)} \ar[r]^\cong & sSet(X,ExSd\, X) \ar[d]^{sSet(id,Ex(d_X))} \\
sSet(Sd\, X,X) \ar[r]_\cong & sSet(X,Ex\, X)
}
\end{displaymath}
in which $d_X$ is sent to $e_X$ under the lower horizontal map by definition and in which the identity is sent to the unit of $(Sd,Ex)$ under the upper horizontal map. The latter square implies that $e_X$ can be obtained by postcomposing the unit with $Ex(d_X)$. The two-out-of-three property implies that the unit is a weak equivalence.

The second map of the composite (\ref{eq:diagram_proof_of_lem_unit_weak_eq}) is the result of applying $Ex^2$ to the unit
\[Sd^2\, X\xrightarrow{\sim } UDSd^2\, X,\]
which is a weak equivalence by \cref{prop:desing_double_subdvision_homotopy_inverse}. Now, the functor $Ex^2$ preserves weak equivalences. This shows that the composite (\ref{eq:diagram_proof_of_lem_unit_weak_eq}) is a weak equivalence.
\end{proof}
\noindent Having proven that the unit of the Quillen pair $(DSd^2,Ex^2U)$ is a weak equivalence is in fact enough, in our case, to prove that the Quillen pair is indeed a Quillen equivalence.

We have so far followed Hirschhorn's terminology throughout this article. However, to prove \cref{prop:homotopy_inverse}, we will use a result in Hovey's book. Hirschhorn's and Hovey's definitions of the term Quillen equivalence are identical to the following.
\begin{definition}
Suppose $F:\mathscr{M} \rightleftarrows \mathscr{N} :G$ a Quillen pair with
\[\varphi :\mathscr{N} (FX,Y)\xrightarrow{\cong } \mathscr{M} (X,GY)\]
the natural bijection that comes with the underlying adjunction $(F,G)$ of categories. We say that $(F,G)$ is a \textbf{Quillen equivalence} if $f:FX\to Y$ is a weak equivalence in $\mathscr{N}$ if and only if $\varphi (f):X\to GY$ is a weak equivalence in $\mathscr{M}$ whenever $X$ is a cofibrant object of $\mathscr{M}$ and $Y$ is a fibrant object of $\mathscr{N}$.
\end{definition}
\noindent Moreover, this definition is independent of any choice of functorial factorizations and any choice of fibrant and cofibrant replacement functors.

A canonical choice of fibrant and cofibrant replacement functors are implicitly part of the model structure in Hovey's notion of model category \cite[Def.~1.1.3, p.~3]{Ho99}, whereas the opposite is true in Hirschhorn's notion \cite[Def.~7.1.3, p.~109]{Hi03}. Namely, Hirschhorn assumes the existence of two functorial factorizations, one as a cofibration followed by a trivial fibration and another as a trivial cofibration followed by a fibration. However, Hovey makes such a choice of functorial factorizations part of the model structure. Thus arises canonical fibrant and cofibrant replacement functors. To think of $(DSd^2,Ex^2U)$ as a Quillen pair according to Hovey, we must then make a choice of functorial factorizations for each of the model categories $sSet$ and $nsSet$.

Now, \cref{thm:lifting_across_adjunction} is the lifting theorem \cite[Thm.~11.3.2]{Hi03}, which applies the recognition theorem \cite[Thm.~11.3.1]{Hi03} whose proof uses the small object argument in the form \cite[Prop.~10.5.16]{Hi03}. From the latter result, which is more or less a standard formulation, we can read off that the small object argument establishes two functorial factorizations on $nsSet$, one into a relative $DSd^2(I)$-cell complex followed by a $DSd^2(I)$-injective, and another into a relative $DSd^2(J)$-cell complex followed by a $DSd^2(J)$-injective. We choose these to serve as part of the model structure on $nsSet$ according to Hovey's notion. Clearly, we follow the same procedure with regards to the sets $I$ and $J$ of maps in $sSet$.

When choices of functorial factorizations have been made, there is a canonical fibrant replacement functor $R$ in $nsSet$ that arises from the factorization
\[A\xrightarrow{r_A} RA\to \Delta [0]\]
of the terminal map, for each non-singular $A$, as a relative $DSd^2(J)$-cell complex $r_A$ followed by a fibration $RA\to \Delta [0]$. In other words, the non-singular simplicial set $A$ is replaced by a fibrant non-singular simplicial set $RA$, with a natural map $r_A$ from the original to its replacement.

The choices of functorial factorizations can simply be forgotten after the proof of \cref{prop:homotopy_inverse}. Because the term Quillen equivalence is defined the same way by both Hirschhorn and Hovey and because this definition has no reference to fibrant or cofibrant replacements, the pair $(DSd^2,Ex^2)$ will be a Quillen equivalence according to Hirschhorn if it is according to Hovey.

Finally, we obtain the last piece used to establish \cref{thm:main_homotopy_theory}, which is the main result.
\begin{proposition}\label{prop:homotopy_inverse}
The Quillen pair
\[DSd^2:sSet\rightleftarrows nsSet:Ex^2U\]
is a Quillen equivalence.
\end{proposition}
\begin{proof}
The pair $(DSd^2,Ex^2U)$ is a Quillen equivalence \cite[Cor.~1.3.16]{Ho99} if and only if $Ex^2U$ reflects weak equivalences between fibrant objects and the composite
\[X\xrightarrow{\eta _X} Ex^2UDSd^2X\xrightarrow{Ex^2U(r_{DSd^2X})} Ex^2URDSd^2X\]
is a weak equivalence for every cofibrant $X$. Here,
\[r_{DSd^2\, X}:DSd^2\, X\xrightarrow{\sim } RDSd^2\, X\]
is the natural relative $DSd^2(J)$-cell complex that comes with the fibrant replacement $R$.

As the model structure on $nsSet$ is lifted along the right adjoint $Ex^2U$, this functor  reflects weak equivalences without an assumption on either the source or the target. For the same reason, the functor $Ex^2U$ preserves weak equivalences. Any object in $sSet$ is cofibrant. Nevertheless, it follows that \cref{prop:homotopy_inverse} holds if the following result holds, which it does.
\end{proof}
\begin{proof}[Proof of \cref{thm:main_homotopy_theory}.]
First, by \cref{prop:main_homotopy_theory}, the category $nsSet$ is a cofibrantly generated model category and $(DSd^2,Ex^2U)$ is a Quillen pair when $sSet$ is equipped with the standard model structure due to Quillen. Second, the model category $nsSet$ satisfies the axiom of propriety according to \cref{prop:axiom_of_propriety}. Finally, \cref{prop:homotopy_inverse} says that the pair $(DSd^2,Ex^2U)$ is a Quillen equivalence.
\end{proof}

%% file: sections/relations.tex
\section{Relating the model categories}
\label{sec:relations}

In this section, we complete the diagram (\ref{eq:diagram_of_adjunctions}) of adjunctions in the sense explained in the introduction. Namely, we promised that the diagram would consist exclusively of model categories and Quillen equivalences.

Verifing that $(D,U)$ is a Quillen equivalence when $sSet$ has the $Sd^2$-model structure of Jardine, is not hard. We state this result as \cref{cor:sd2_structure_d_u_Quillen}. Similarly, we can verify that $(q,N)$ is a Quillen equivalence when $PoSet$ has the model structure of Raptis. This we state as \cref{cor:square_of_quillen_equivalences}.

First, we establish the relationship with posets.
\begin{lemma}\label{cor:square_of_quillen_equivalences}
If $PoSet$ has Raptis' model structure \cite{Ra10} and $nsSet$ has the model structure suggested in \cref{thm:main_homotopy_theory}, then $(q,N)$ is a Quillen equivalence.
\end{lemma}
\begin{proof}
A set of generating cofibrations in Thomason's model category $Cat$ is $cSd^2(I)$ and a set of generating trivial cofibrations is $cSd^2(J)$, as Raptis points out in his overview and slight modernization of Thomason's work \cite[Thm.~2.2, p.~215]{Ra10}.

Raptis' cofibrantly generated model structure on $PoSet$ is restricted from $Cat$ in the sense that the weak equivalences of $PoSet$ are the weak equivalences of $Cat$ whose source and target are both posets, and similarly for the cofibrations and the fibrations \cite[Thm.~2.6 ,p.~217]{Ra10}. The sets $pcSd^2(I)$ and $pcSd^2(J)$ can be taken to be a set of generating cofibrations and a set of generating trivial cofibrations in $PoSet$ as well, respectively \cite[Thm.~2.6, p.~217]{Ra10}.

Consider applying the functor
\[q:nsSet\to PoSet\]
to the class $DSd^2(I)-cof$ of cofibrations in $nsSet$. The functor $q$ is in \cref{sec:intro_hty} defined as $q=pcU$. Due to the equality $N\circ U=U\circ N$ of the two composites of right adjoints and by the uniqueness of the left adjoint, we get a natural isomorphism $pc\, X\xrightarrow{\cong } qD\, X$. Thus we get the equality in the expression
\[q(DSd^2(I)-cof)\subseteq qDSd^2(I)-cof=pcSd^2(I)-cof\]
where the inclusion comes from a general rule stated as Lemma 2.1.8 in \cite[p.~30]{Ho99}. Hence, the left adjoint $q$ preserves cofibrations. Similarly, by replacing $I$ by $J$, we see that $q$ preserves the trivial cofibrations. This finishes our verification that $q$ is a left Quillen functor and hence that $(q,N)$ is a Quillen pair.

The composite of $(p,U)$ and $(cSd^2,Ex^2N)$ is a Quillen equivalence. Furthermore, the composite of $(q,N)$ and $(DSd^2,Ex^2U)$ is a Quillen pair. By Corollary 1.3.14 in \cite[p.~20]{Ho99}, the latter composite is a Quillen equivalence if and only if the former is. Now, consider the two Quillen pairs $(q,N)$, $(DSd^2,Ex^2U)$ together with their composite. By \cref{thm:main_homotopy_theory} we know that two of these three Quillen pairs are Quillen equivalences. Hence, the third is a Quillen equivalence by Corollary 1.3.15. in Hovey's book \cite[p.~21]{Ho99}.
\end{proof}
\noindent Finally, we establish the relationship with Jardine's $Sd^2$-model structure on simplicial sets.
\begin{lemma}\label{cor:sd2_structure_d_u_Quillen}
Let the category $sSet$ have J. F. Jardine's $Sd^2$-structure from \cite[p.~274]{Ja13}. Then $(D,U)$ is a Quillen equivalence.
\end{lemma}
\begin{proof}
As in the proof of \cref{cor:square_of_quillen_equivalences}, we need only prove that $(D,U)$ is a Quillen pair. Then, by the two out of three-property for Quillen equivalences, it will follow that $(D,U)$ is a Quillen equivalences as $(Sd^2,Ex^2)$ is a Quillen equivalence according to J. F. Jardine \cite[Thm.~1.1.,~p.~274]{Ja13} and as $(DSd^2,Ex^2U)$ is a Quillen equivalence according to \cref{thm:main_homotopy_theory}.

We verify that $U$ is a right Quillen functor by verifying that it preserves fibrations and trivial fibrations. Then $(D,U)$ will be a Quillen pair. First, if $f$ is a fibration in $nsSet$, then $Ex^2Uf$ is a Kan fibration, by definition. Thus $Uf$ is an $Ex^2$-fibration by definition.

Second, if $f$ is a trivial fibration in $nsSet$, then $f$ is by definition both a weak equivalence in $nsSet$ and a fibration in $nsSet$. Thus $Uf$ is an $Ex^2$-fibration by the previous paragraph. Furthermore, the map $Ex^2Uf$ is a weak equivalence by definition. As $Ex$ preserves and reflects weak equivalences, it follows that $Uf$ is a weak equivalence. Recall that the weak equivalences in the standard model structure and the $Sd^2$-model structure are the same. Hence, $Uf$ is a trivial $Ex^2$-fibration. This concludes our verification that $U$ is a right Quillen functor.
\end{proof}